\documentclass[a4paper,11pt]{amsart}

\usepackage{amssymb,amsthm,amsmath,amscd,amsfonts,bbm,mathrsfs,mathdots}
\usepackage{hyperref}
\usepackage{enumitem}
\usepackage{xcolor}

\usepackage[cmtip,all]{xy}

\theoremstyle{plain}
\newtheorem{Lemma}{Lemma}
\newtheorem*{Lemma*}{Lemma}
\newtheorem{Thm}[Lemma]{Theorem}
\newtheorem*{Thm*}{Theorem}
\newtheorem{Prop}[Lemma]{Proposition}
\newtheorem{Cor}[Lemma]{Corollary}
\theoremstyle{definition}
\newtheorem{Defn}[Lemma]{Definition}
\newtheorem{Example}[Lemma]{Example}
\newtheorem{Const}[Lemma]{Construction}
\theoremstyle{remark}
\newtheorem{Remark}[Lemma]{Remark}
\newtheorem{Point}[Lemma]{}

\numberwithin{Lemma}{subsection}
\numberwithin{equation}{section}

\newcommand{\CCC}{\mathcal{C}}
\newcommand{\DDD}{\mathcal{D}}
\newcommand{\FFF}{\mathcal{F}}

\newcommand{\MMM}{\mathcal{M}}
\newcommand{\NNN}{\mathcal{N}}
\newcommand{\OOO}{\mathcal{O}}

\newcommand{\SSS}{\mathcal{S}}
\newcommand{\UUU}{\mathcal{U}}

\newcommand{\Fa}{\mathfrak{a}}

\newcommand{\Fg}{\mathfrak{g}}
\newcommand{\Fm}{\mathfrak{m}}

\renewcommand{\AA}{{\mathbb{A}}}
\newcommand{\BB}{{\mathbb{B}}}

\newcommand{\FF}{{\mathbb{F}}}
\newcommand{\GG}{{\mathbb{G}}}
\newcommand{\NN}{{\mathbb{N}}}
\newcommand{\II}{{\mathbb{I}}}
\newcommand{\QQ}{{\mathbb{Q}}}
\newcommand{\WW}{{\mathbb{W}}}
\newcommand{\ZZ}{{\mathbb{Z}}}

\newcommand{\bA}{{\mathbf A}}
\newcommand{\bB}{{\mathbf B}}
\newcommand{\bM}{{\mathbf M}}
\newcommand{\bN}{{\mathbf N}}
\newcommand{\bS}{{\mathbf S}}

\newcommand{\Gm}{{\mathbb{G}_m}}

\DeclareMathOperator{\Ad}{Ad}
\DeclareMathOperator{\AffNilp}{AffNilp}

\DeclareMathOperator{\Aut}{Aut}

\DeclareMathOperator{\End}{End}

\DeclareMathOperator{\Frob}{Frob}
\DeclareMathOperator{\GDisp}{-Disp}
\DeclareMathOperator{\GL}{GL}

\DeclareMathOperator{\Hom}{Hom}

\DeclareMathOperator{\Image}{Im}
\DeclareMathOperator{\Isom}{Isom}
\DeclareMathOperator{\Ker}{Ker}
\DeclareMathOperator{\Lie}{Lie}
\DeclareMathOperator{\Max}{Max}
\DeclareMathOperator{\Mod}{Mod}

\DeclareMathOperator{\Ogr}{O}
\DeclareMathOperator{\Rad}{Rad}
\DeclareMathOperator{\Spec}{Spec}
\DeclareMathOperator{\Spf}{Spf}
\DeclareMathOperator{\Sym}{Sym}

\DeclareMathOperator{\uAut}{\underline{Aut}}
\DeclareMathOperator{\uHom}{\underline{Hom}}
\DeclareMathOperator{\uIsom}{\underline{Isom}}

\DeclareMathOperator{\crys}{crys}
\DeclareMathOperator{\diag}{diag}

\DeclareMathOperator{\gr}{gr}

\DeclareMathOperator{\id}{id}
\DeclareMathOperator{\iso}{iso}

\DeclareMathOperator{\rk}{rk}
\DeclareMathOperator{\tr}{tr}

\newcommand{\fgauge}{\mathsf f}
\newcommand{\vgauge}{\mathsf v}

\renewcommand{\u}{\underline}
\newcommand{\uW}[1]{\u{W}{}_{#1}}

\begin{document}

\title{Higher frames and $G$-displays}
\author{Eike Lau}
\date{\today}

\begin{abstract}
Deformations of ordinary varieties of K3 type can be described in terms of displays by recent work of Langer-Zink. We extend this to the general (non-ordinary) case using displays with $G$-structure for a reductive group $G$. As a basis we suggest a modified definition of the tensor category of displays and variants which is similar to the Frobenius gauges of Fontaine-Jannsen. 
\end{abstract}

\maketitle

\section{Introduction}

Let $p$ be a prime.  

If $R$ is a ring in which $p$ is nilpotent,
by Langer-Zink \cite{Langer-Zink:DRW-Disp} the crystalline
cohomology of a smooth projective scheme over $R$ with good properties
carries the structure of a display\footnote{To 
simplify the terminology we write display instead of higher display,
while the 3n-displays of \cite{Zink:Display} will be called $1$-displays.} 
over the ring of Witt vectors $W(R)$.
In the case of abelian varieties this is related to the natural structure of
a 1-display
on the crystalline Dieudonn\'e module of a $p$-divisible group.

The notion of displays admits many variations.
For a local Artin ring $R$ with perfect residue field of odd characteristic
one can consider displays over Zink's small Witt vector ring,
denoted in the following by $\WW(R)$, as in
\cite{Zink:Dieudonne, Lau:Relations, Langer-Zink:K3}.
There is a definition of 1-displays over $m$-truncated Witt vectors,
which are related to truncated $p$-divisible groups
by \cite{Lau:Smoothness, Lau-Zink:Truncated}, 
and which are related to the $F$-zips of \cite{Moonen-Wedhorn:Discrete}
when $m=1$.

One can also endow displays and their variants with a $G$-structure
for a reductive group $G$. 
For $F$-zips this is carried out in \cite{Pink-Wedhorn-Ziegler:F-zips}.
A definition of $1$-displays over $W(R)$ with $G$-structure 
for a smooth group scheme appears in \cite{Bueltel}.
Building on this, B\"ultel-Pappas \cite{Bueltel-Pappas} defined $(G,\mu)$-displays over $W(R)$
for a minuscule cocharacter $\mu$ 
and used these to construct Rapoport-Zink spaces of Hodge type.
Here an essential point is a deformation theory similar
to the classical deformation theory of $1$-displays of \cite{Zink:Display}.


%
%
%

In this article we introduce a modified definition of displays in an abstract
setting that allows to treat different display-like objects uniformly, and
that allows a general definition of $G$-displays for a smooth group scheme
$G$ over $\ZZ_p$ relative to an arbitrary cocharacter $\mu$,
which have a good deformation theory when $\mu$ is minuscule.
As an application, a result of \cite{Langer-Zink:K3} 
on deformations of ordinary schemes of K3 type in terms of
displays over $\WW(R)$ with a quadratic form
is extended to the general (non-ordinary) case.

\medskip

Let us explain the content of this article in more detail.

As a basis of our definition of displays we consider triples 
$\u S=(S,\sigma,\tau)$,
called frames, which consist of a $\ZZ$-graded ring $S=\bigoplus S_n$
and ring homomorphisms $\sigma,\tau:S\to S_0$ with certain axioms.
A predisplay over $\u S$ is a graded $S$-module $M$ 
with a homomorphism of $S_0$-modules $F:M^\sigma\to M^\tau$.
A predisplay is called a display if $M$ is finite projective
and $F$ is bijective.
The tensor product of graded modules gives a tensor product of predisplays 
which preserves displays.
In this way, the exact tensor category of displays is embedded into the
abelian tensor category of predisplays.

This definition of displays has some similarity 
with the theory of Frobenius gauges of 
\cite{Fontaine-Jannsen:Frobenius-gauges};
from that point of view the main question is to find the correct 
$\varphi$-ring in the category of $\varphi$-gauges. 
But we will also consider frames that are not $\varphi$-gauges,
in particular when base rings occur which are not $\FF_p$-algebras.


Examples of frames:

(1)
For a $p$-adic ring $R$ one can define a Witt frame $\u W(R)$
such that the resulting displays are equivalent to
the displays of Langer-Zink (but the predisplays are different).
When $R$ is a perfect $\FF_p$-algebra, the corresponding description of Langer-Zink displays by $\varphi$-gauges appears in \cite{Wid}.

(2)
For an $\FF_p$-algebra $R$ there is a truncated version $\uW m(R)$ of the Witt frame,
which gives a definition of (higher) truncated displays.
For $m=1$ these are equivalent to the $F$-zips of \cite{Moonen-Wedhorn:Discrete}.
This is similar to a result of \cite{Schnellinger} showing that $\varphi$-$R$-gauges
are equivalent to the modified $F$-zips of \cite{Wedhorn:De-Rham}.

(3)
A local Artin ring $R$ with perfect residue field of odd characteristic will be called \emph{admissible}. In this case, the small Witt vector ring gives rise to a frame $\u\WW(R)$,
which gives the displays over $\WW(R)$ of \cite{Langer-Zink:K3}.

(4)
There are relative versions associated to divided power thickenings
of the Witt vector frame and its small variant,
denoted by $\u W(B/A)$ and $\u\WW(B/A)$,
which control the deformation theory of displays.

\medskip

Now let $G=\Spec A$ be a smooth affine group scheme over $\ZZ_p$ 
and let $\mu:\Gm\to G$ be a cocharacter;
later we allow that $\mu$ is
defined over a finite unramified extension of $\ZZ_p$.
The conjugation action 
via $\mu^{-1}$ makes $A$ into a $\ZZ$-graded ring. 
For a frame $\u S$ we define the display group
\[
G(S)_\mu\subseteq G(S)
\]
as the subgroup of graded ring homomorphisms $A\to S$,
or equivalently of invariant sections with respect to
the $\Gm$-action on $\Spec S$ corresponding to the given grading of $S$.
The ring homomorphisms $\sigma,\tau$ induce group
homomorphism $G(S)_\mu\to G(S_0)$, and we consider the right action
\[
G(S_0)\times G(S)_\mu \to G(S_0),\qquad (g,h)\mapsto \tau(h)^{-1}g\sigma(h).
\]
If the frame $\u S$ is functorial with respect to etale homomorphisms
of $R$, which holds for all frames which are related to Witt vectors,
we define the groupoid of $G$-displays of type $\mu$ over $\u S$ as the
quotient groupoid of this action with respect of the etale topology of $R$.
For $\u S=\uW 1(R)$ this gives the $G$-zips of type $\mu^{-1}$ of
\cite{Pink-Wedhorn-Ziegler:F-zips}; see Example \ref{Ex:G-zip}.
For $\u S=\u W(R)$ and minuscule $\mu$ this gives the
$(G,\mu^{-1})$-displays of \cite{Bueltel-Pappas}; see Remark \ref{Rk:BP}.

For suitable frames
there is a good deformation theory for $G$-displays of minuscule type $\mu$.
The key result, stated in the context of displays over $\WW(R)$, is the following proposition (see Proposition \ref{Pr:deformation-lemma} and Example \ref{Ex:WBA-sigma-nil-thickening}).

\begin{Prop}
\label{Pr:Intro}
For a nilpotent divided power thickening $B\to A$ 
of admissible local Artin rings
the frame homomorphism $\u\WW(B/A)\to\u\WW(A)$ induces an equivalence 
of $G$-displays of minuscule type $\mu$.
\end{Prop}

We have the following geometric application.
Let $R$ be an admissible local Artin ring with residue field $k$.
Langer-Zink \cite{Langer-Zink:K3}
consider a class of schemes $X/R$ of K3 type, which includes K3 surfaces,
and show that their second crystalline cohomology 
carries a display structure over $\u\WW(R)$ 
and a perfect quadratic form, which is the cup product in the case of K3 surfaces.
The following is proved in \cite{Langer-Zink:K3} when $X_0$ is ordinary.

\begin{Thm}
\label{Th:Intro}
(Theorem \ref{Th:Def-K3-scheme})
For any scheme of K3 type $X_0/k$ 
the deformations over $R$ of $X_0$ 
and of its second crystalline cohomology
as a display with a quadratic form
are equivalent. 
\end{Thm}

Let us briefly explain the new aspect of the proof given here.
The Frobenius type of the second crystalline cohomology of a scheme of
K3 type corresponds (after a twist) to the cocharacter $\mu=(1,0,\ldots,0,-1)$
of the group $\GL_n$, which is not minuscule, 
so the deformation theory of higher displays of type $\mu$ does not work well.
But $\mu$ is minuscule as a cocharacter 
of a suitable orthogonal group $G\subset\GL_n$,
and $G$-displays correspond to displays with a non-degenerate quadratic form.
Therefore Proposition \ref{Pr:Intro} can be applied.

\subsection*{Terminology and Notation.}

We fix a prime $p$.
All rings are commutative with a unit.
A ring or an abelian group is called $p$-adic 
if it is $p$-adically complete and separated.
All divided power thickenings (PD thickenings) of rings will be assumed
to be compatible with the natural divided powers of $p$.
An ideal $I$ of a ring $R$ is called bounded nilpotent if there is an $n$
such that $x^n=0$ for all $x\in I$.
An endomorphism $f$ of an abelian group $A$ is called pointwise nilpotent
if for each $a\in A$ there is an $n$ with $f^n(a)=0$. 
If $u:R\to S$ is a ring homomorphism, for an $R$-module $M$
we write $M^u=M\otimes_{R,u}S$. 
If $\alpha:G\to H$ is a homomorphism of groups in a topos and
$P$ is a (right) $G$-torsor we denote by $P^\alpha$ the associated $H$-torsor.
%


\setcounter{tocdepth}{1}
\tableofcontents

\section{Higher frames}

Let $p$ be a prime. In the following, frame means higher frame, while the frames of \cite{Zink:Windows, Lau:Frames} and their variants will be called 1-frames.


\begin{Defn}
\label{Def:frame}
A pre-frame $\u S=(S,\sigma,\tau)$ consists of a $\ZZ$-graded ring 
\[
S=\bigoplus_{n\in\ZZ} S_n
\]
and ring homomorphisms $\sigma,\tau:S\to S_0$.
Let $\sigma_n,\tau_n:S_n\to S_0$ denote the restrictions of $\sigma,\tau$.
The pre-frame $\u S$ is called a frame if the following holds.
\begin{enumerate}
\item
\label{It:frame-a}
The endomorphism $\tau_0$ is the identity of $S_0$, 
and $\tau_{-n}:S_{-n}\to S_0$ is bijective for $n\ge 1$.
Let $t\in S_{-1}$ be the unique element with $\tau_{-1}(t)=1$.
\item
\label{It:frame-b}
The endomorphism $\sigma_0$ of $S_0$ is a Frobenius lift,
and $\sigma_{-1}(t)=p$.
\item
\label{It:frame-c}
We have $p\in\Rad(S_0)$, the Jacobson radical of $S_0$.
\end{enumerate}
We say that $\u S$ is a frame for $R=S_0/tS_1$.
A homomorphism of pre-frames is a homomorphism of graded rings 
that commutes with $\sigma$ and $\tau$.
\end{Defn}



\begin{Remark} 
\label{Rk:frame-recover}
The frame axioms imply that $S_{\le 0}=S_0[t]$ is the polynomial ring
with variable $t$.
A frame $\u S$ is determined by the graded ring $S_{\ge 0}$ 
together with the ring homomorphism $\sigma:S_{\ge 0}\to S_0$ 
and the maps $t:S_{n+1}\to S_{n}$ for $n\ge 0$ 
subject to the following axioms.
\begin{enumerate}
\item
\label{It:frame-recover a}
The homomorphism $t:S_{\ge 1}\to S_{\ge 0}$ is $S_{\ge 0}$-linear.
\item
\label{It:frame-b}
$\sigma_0:S_0\to S_0$ is a Frobenius lift, and
$\sigma_n(t(a))=p\sigma_{n+1}(a)$ for $a\in S_{n+1}$.
\item
\label{It:frame-c}
We have $p\in\Rad(S_0)$.
\end{enumerate}
One recovers $\tau_n(a)=t^n(a)$.
In the examples of frames we will often define the triple
$(S_{\ge 0},\sigma,t)$,
but the definition of displays over a frame is more easily
expressed in terms of the data $(S,\sigma,\tau)$.
\end{Remark}

\begin{Remark}
\label{Rk:tau}
If $\u S$ is a frame for $R$,
the homomorphism 
$\tau$ induces an isomorphism $S/(t-1)S\cong R$.
\end{Remark}

\begin{Remark}
\label{Rk:S0pRp-nil}
If $\u S$ is a frame for $R$,
each $a$ in the kernel of $S_0/p\to R/p$ satisfies $a^p=0$
since $\sigma(ty)=p\sigma_1(y)$ for $y\in S_1$.
\end{Remark}


%

\begin{Remark}
\label{Rk:O-frames}
Let $\OOO$ be the ring of integers in a finite extension of $\QQ_p$ 
with residue field $\FF_q$ and with a fixed prime element $\pi\in\OOO$.
One can define $\OOO$-frames by the following modification 
of Definition \ref{Def:frame}: $S$ is a graded $\OOO$-algebra, 
$\sigma_0$ is a lift of the $q$-Frobenius,
and $\sigma_{-1}(t)=\pi$.
This will not be carried out in the following, 
but see Remarks \ref{Rk:frames-WO} \& \ref{Rk:G-disp-O}
and Example \ref{Ex:G-zip}.
\end{Remark}

We refer to \S\ref{Subse:DisplayGauge} for the relation between frames and $\varphi$-gauges.

\subsection{Examples of frames}

\begin{Example}[Frame associated to a filtration with Frobenius divisibility]
\label{Ex:frame-torsion-free}
If $A$ is a $p$-torsion free ring with $p\in\Rad(A)$,
equipped with a Frobenius lift $\sigma_0:A\to A$
and a descending sequence of ideals
$A=J_0\supseteq J_1\supseteq J_2\supseteq\ldots$ 
such that $J_nJ_m\subseteq J_{n+m}$ and $\sigma_0(J_n)\subseteq p^nA$,
we obtain a frame $\u S$ for $R=A/J_1$ as follows.
For $n\ge 0$ we set $S_n=J_n$ and $\sigma_n=p^{-n}\sigma_0$, 
and $t:S_{n+1}\to S_n$ is the inclusion.
This determines $\u S$ by Remark \ref{Rk:frame-recover}.

Explicitly, $S$ is the ring of all Laurent polynomials 
$\sum a_{n}t^{-n}\in A[t,t^{-1}]$ 
such that $a_{n}\in J_n$ for $n\ge 1$, 
with the structure of a $\ZZ$-graded $A$-algebra
such that $\deg(t)=-1$.
The homomorphisms $\sigma,\tau:S\to S_0=A$
extend $\sigma_0$ and $\tau_0=\id$ by $\sigma(t)=p$ and $\tau(t)=1$.
\end{Example}

\begin{Example}[Tautological frame]
\label{Ex:frame-tautological}
Let $A$ be a ring with $p\in\Rad(A)$ 
and with a Frobenius lift $\sigma_0:A\to A$. 
There is a unique frame $\u S$ with $S_0=A$ 
and $S_{n}=0$ for $n>0$ such that
$\sigma$ extends the given $\sigma_0$,
namely $S=A[t]$.
\end{Example}

For a ring $R$ let $W(R)$ be the ring of $p$-typical Witt vectors
with Frobenius $\sigma$ and Verschiebung $v$, and let $I(R)=v(W(R))$ so that $W(R)/I(R)\cong R$.

\begin{Example}[Witt frame]
\label{Ex:frame-W}
For a $p$-adic ring $R$ we define the Witt frame $\u S=\u W(R)$ as follows,
using again Remark \ref{Rk:frame-recover}.
We set $S_0=W(R)$, and $S_n=I(R)=v(W(R))$ as an $S_0$-module for $n\ge 1$.
For $n,m\ge 1$ the multiplication map $S_n\times S_m\to S_{n+m}$
is given by 
\[
I(R)\times I(R)\to I(R),\qquad (v(a), v(b))\mapsto v(ab).
\]
The homomorphism $t:S_1\to S_0$ is the inclusion $I(R)\to W(R)$,
and $t:S_{n+1}\to S_n$ is the multiplication $p:I(R)\to I(R)$ for $n\ge 1$. 
The endmorphism $\sigma_0$ of $W(R)$ is the Witt vector Frobenius,
and we set $\sigma_n(v(a))=a$ for $n\ge 1$.
The axioms of Remark \ref{Rk:frame-recover} are easily verified;
the condition $p\in\Rad(W(R))$ holds since $W(R)$ is $p$-adic
by \cite[Prop.\ 3]{Zink:Display}.
$\u W(R)$ is a frame for $R$, which is functorial in $R$.
\end{Example}

\begin{Remark}
The pre-frame $\u W(R)$ can be defined for an arbitrary ring $R$,
but the condition $p\in\Rad(R)$ does not hold in general.
%
\end{Remark}

\begin{Remark}
\label{Rk:frame-W-tf}
If $R$ is a $p$-adic ring such that $W(R)$ is $p$-torsion free,
the Witt frame $\u W(R)$ coincides with the frame 
of Example \ref{Ex:frame-torsion-free} 
for $A=W(R)$ and $J_n=p^{n-1}I(R)=I(R)^n$. 
The last equation holds since $v(a)v(b)=pv(ab)$.
\end{Remark}

\begin{Example}[Truncated Witt frame]
\label{Ex:frame-Wm}
For an $\FF_p$-algebra $R$ and $m\ge 1$
we can define a truncated Witt frame $\u S=\uW m(R)$
by $S_0=W_m(R)$ and $S_n=I_{m+1}(R)=v(W_{m+1}(R))$ for $n\ge 1$,
with the unique frame structure such that the natural surjective map 
$\u W(R)\to\uW m(R)$ is a homomorphism of frames.
\end{Example}

\begin{Example}[Zip frame]
\label{Ex:frame-W1}
For an $\FF_p$-algebra $R$, the frame $\u S=\uW 1(R)$ 
will also be called the zip frame of $R$;
this terminology is explained by Example~\ref{Ex:disp-W1-zip} below.
Explicitly $\u S$ looks as follows.
We have $S_n=R$ for all $n$, thus $S=R^{(\ZZ)}$ as an abelian group.
The multiplication $S_n\times S_m\to S_{n+m}$ is zero if $n<0$ and $m>0$,
it is the usual multiplication of $R$ if $n,m\le 0$ or $n,m\ge 1$, 
and it is given by $(a,b)\mapsto a^pb$ if $n=0$ and $m\ge 1$.
We can identify
\begin{equation}
\label{Eq:W1}
S=R[t]\times_{\Frob,R}R[u]=\{(f,g)\in R[t]\times R[u]\mid f(0)^p=g(0)\}
\end{equation}
by sending $(a)\in R^{(\ZZ)}$ to $f=\sum_{n\ge 0}a_{-n}t^n$ and
$g=a_0^p+\sum_{n\ge 1}a_nu^n$. 
Then $\sigma,\tau:S\to R$ are given by 
$\tau((f,g))=f(1)$ and $\sigma((f,g))=g(1)$.
\end{Example}

\begin{Remark}
\label{Rk:frame-W1-final}
For an $\FF_p$-algebra $R$,
the zip frame $\uW 1(R)$ is a final object in the category of frames for $R$.
Namely, let $\u S$ be a frame for $R$ and let $\pi_0:S_0\to R$ 
denote the projection.
For $n\ge 1$ define $\pi_{-n}:S_{-n}\to R$ and $\pi_n:S_n\to R$ by 
$\pi_{-n}(a)=\pi_0(\tau_{-n}(a))$
and $\pi_n(b)=\pi_0(\sigma_n(b))$.
Then $\pi$ is the unique frame homomorphism $\u S\to\uW 1(R)$
over the identity of $R$.
Similarly, if $\u S'$ is a frame for a ring $R'$
then $\Hom(R',R)=\Hom(\u S',\uW 1(R))$.
\end{Remark}

\begin{Example}[Relative Witt frame]
\label{Ex:frame-W-rel}
For a PD thickening of $p$-adic rings $B\to A=B/J$
we define the relative Witt frame $\u S=\u W(B/A)$ as follows,
using again Remark \ref{Rk:frame-recover}.
Set $S_0=W(B)$ and 
$S_n=I(B)\oplus J$ as an $S_0$-module for $n\ge 1$,
where $S_0$ acts on $J$ by the projection $S_0\to B$.
For $n,m\ge 1$, the multiplication $S_n\times S_m\to S_{n+m}$ 
is defined by
\[
(v(a)+x)(v(b)+y)=v(ab)+xy
\]
for $a,b\in W(B)$ and $x,y\in J$.
We recall 
that the divided powers on $J$ induce an isomorphism 
\begin{equation}
\label{Eq:log}
\log:W(J)\to J^\NN
\end{equation}
as in \cite[\S 1.4]{Zink:Display}.
Define
\[
t:S_1\to S_0,\qquad
t(v(a)+x)=v(a)+\log^{-1}(x,0,0,\ldots)
\]
and
\[
t:S_n\to S_{n-1},\qquad
t(v(a)+x)= pv(a)+x
\]
for $n\ge 2$. 
We note that $t:S_1\to I(B/A)=\Ker(W(B)\to A)$ is bijective.
The endomorphism $\sigma_0$ of $W(B)$ is the Witt vector Frobenius,
and for $n\ge 1$ we set $\sigma_n(v(a)+x)=a$ for $a\in W(B)$ and $x\in J$.
The axioms of Remark \ref{Rk:frame-recover} are easily verified.
$\u W(B/A)$ is a frame for $A$, which is functorial in 
the PD thickening $B\to A$.
For $A=B$ we have $\u W(B/A)=\u W(A)$.
\end{Example}

\begin{Remark}
If $W(B)$ is $p$-torsion free, the frame $\u W(B/A)$ coincides with the
frame of Example \ref{Ex:frame-torsion-free} 
for $A=W(B)$ and $J_n=p^{n-1}I(B)\oplus J$
where $J$ is viewed as an ideal of $W(B)$ 
by $x\mapsto\log^{-1}(x,0,0,\ldots)$;
cf.\ Remark \ref{Rk:frame-W-tf}.
\end{Remark}

\begin{Remark}
If $B\to A$ is a PD thickening of $\FF_p$-algebras,
one can also define a truncated version $\uW m(B/A)$
of the relative Witt frame.
\end{Remark}

\begin{Remark}
[Minimal frames]
\label{Rk:Smin}
Let us call a frame $\u S$ minimal if $S_{\ge 0}$ is generated by $S_1$ 
as an $S_0$-algebra. 
For any frame $\u S$ we define a sub-frame $\u S'=\u S^{\min}\subseteq\u S$ as follows.
For $n\le 1$ let $S'_n=S_n$ and for $n\ge 2$ let $S'_n\subseteq S_n$ 
be the image of the multiplication map $S_1\otimes\ldots\otimes S_1\to S_n$.
The functor
$\u S\mapsto\u S^{\min}$ is right adjoint to the inclusion functor from minimal frames to frames.
We have $\u W(R)=\u W(R)^{\min}$, but $\u W(B/A)\ne\u W(B/A)^{\min}$ in general.
\end{Remark}


\begin{Example}
[Zink frame]
\label{Ex:frame-WW}
Let $R$ be a local ring such that $\Fm_R$ is bounded nilpotent and
the residue field $k=R/\Fm$ is perfect of characterstic $p$.
Zink \cite{Zink:Dieudonne} defines a subring 
$\WW(R)=W(k)\oplus\hat W(\Fm)$ of $W(R)$,
here we use the notation of \cite{Lau:Relations}.
We call the ring $R$ \emph{admissible} if $p\ge 3$ or $pR=0$;
this is more general than in the introduction.
In this case the ring $\WW(R)$ gives rise to a subframe 
$\u\WW(R)$ of $\u W(R)$
consisting of $\WW(R)\subseteq W(R)$ in degree $0$ and 
$\II(R)=\WW(R)\cap I(R)\subseteq I(R)$ in each degree $n\ge 1$.
The admissibility implies that $\sigma_1:\II(R)\to W(R)$
has image in $\WW(R)$ as required.
We have $p\in\Rad(\WW(R))$ since $\WW(R)$ is $p$-adic;
see \cite[\S 1F]{Lau:Relations}.
The frame $\u\WW(R)$ is functorial in $R$.
\end{Example}

\begin{Example}[Relative Zink frame]
\label{Ex:frame-WW-rel}
If $B\to A=B/J$ is a PD thickening of admissible rings 
as in Example \ref{Ex:frame-WW}
with nilpotent divided powers, there is a subframe
$\u\WW(B/A)$ of $\u W(B/A)$ consisting of $\WW(B)\subseteq W(B)$
in degree $0$ and $\II(B)\oplus J\subseteq I(B)\oplus J$ 
in each degree $n\ge 1$.
Indeed, for nilpotent divided powers the isomorphism 
$\log$ of \eqref{Eq:log} induces an isomorphism 
$\hat W(J)\cong J^{(\NN)}$,
so the homomorphism
$t:\II(B)\oplus J\to W(B)$ has image in $\WW(B)$ as required.
The frame $\u\WW(B/A)$ is functorial in the PD thickening $B\to A$.
For $A=B$ we have $\u\WW(B/A)=\u\WW(A)$.
\end{Example}

\begin{Remark}
\label{Rk:frames-WO}
Let $\OOO$, $q$, $\pi$ be as in Remark \ref{Rk:O-frames}.
All frames related to the Witt vectors have obvious variants
for Drinfeld's $\OOO$-Witt vectors $W_\OOO$ \cite{Drinfeld:Coverings}.
More precisely, 
for a $\pi$-adic $\OOO$-algebra there is an $\OOO$-frame $\u W{}_\OOO(R)$, 
for an $\OOO$-PD thickening of $\pi$-adic $\OOO$-algebras
$B\to A$ there is an $\OOO$-frame $\u W{}_\OOO(B/A)$,
for an $\FF_q$-algebra $R$ there are truncated $\OOO$-frames $\uW {\OOO,m}(R)$,
and for a local Artin $\OOO$-algebra $R$ there is an $\OOO$-frame
$\u\WW{}_\OOO(R)$, with a relative version for a nilpotent $\OOO$-PD thickening. 
For an $\FF_q$-algebra $R$ the $q$-zip frame $\uW {\OOO,1}(R)$
is described as in Example \ref{Ex:frame-W1} with $q$ in place of $p$.
\end{Remark}

\subsection{1-frames and Verj\"ungung}

One source of (higher) frames are $1$-frames with Verj\"ungung as
defined in \cite{Langer-Zink:K3} with some mild modifications.

\begin{Defn}
\label{Def:1-frame}
A $1$-frame $\SSS=(S_0\supset I,\sigma_0,\dot\sigma)$ consists of a ring $S_0$,
an ideal $I\subset S_0$, 
a ring endomorphism $\sigma_0$ of $S_0$, 
and a $\sigma_0$-linear map $\dot\sigma:I\to S_0$
with the following properties.
\begin{enumerate}
\item
$\sigma_0:S_0\to S_0$ is a Frobenius lift.
\item
$\sigma_0(a)=p\dot\sigma(a)$ for $a\in I$.
\item
$p\in\Rad(S_0)$.
\end{enumerate}
We say that $\SSS$ is a $1$-frame for $R=S_0/I$.
\end{Defn}

A $1$-frame 
is a frame in the sense of \cite{Lau:Frames} with $\theta=p$ 
but without surjectivity condition on $\dot\sigma$; 
see \cite[Remark 2.11]{Lau:Frames}.
We note that $I\subseteq\Rad(S_0)$ as required in 
\cite[Def.\ 2.1]{Lau:Frames}
since the existence of $\dot\sigma$ implies that
the image of $I\to S_0/p$ is a nil-ideal.

For a frame $\u S$ such that $t:S_1\to S_0$ is injective,
a $1$-frame $\SSS$ is defined by $I=tS_1$ and 
$\dot\sigma(ta)=\sigma_1(a)$ for $a\in S_1$.
We say that the frame $\u S$ extends the $1$-frame $\SSS$.

The following is a variant of a definition in \cite{Langer-Zink:K3}.

\begin{Defn}
\label{Def:Verj}
Let $\SSS$ be a 1-frame. 
A structure of Verj\"ungung on $\SSS$ is a pair $(\nu,\pi)$ where
$\nu:I\otimes_{S_0}I\to I$ and $\pi:I\to I$ are $S_0$-linear maps 
such that the following holds. We write $\nu(a\otimes b)=a*b$.
\begin{enumerate}
\item
$\nu$ is commutative and associative,
\item
$\pi(a*b)=ab=\pi(a)*b$ for $a,b\in I$,
\item
$\dot\sigma(a*b)=\dot\sigma(a)*\dot\sigma(b)$ for $a,b\in I$,
\item
$\dot\sigma(\pi(a))=\sigma_0(a)$ for $a\in I$.
\end{enumerate}
\end{Defn}

\begin{Const}
\label{Const:1fr2fr}
Let $\SSS=(S_0\supset I,\sigma_0,\dot\sigma)$ 
be a 1-frame with Verj\"ungung $(\nu,\pi)$.
We construct a frame $\u S$ that extends $\SSS$,
using Remark \ref{Rk:frame-recover}.
Necessarily $t:S_1\to S_0$ is the inclusion map $I\to S_0$.
For $n\ge 2$ we set $S_n=I$ as an $S_0$-module, 
and $t:S_n\to S_{n-1}$ is equal to $\pi$.
For $n,m\ge 1$ the multiplication $S_n\times S_m\to S_{n+m}$ 
is given by $\nu$,
and $\sigma_n(a)=\dot\sigma(a)$ for $n\ge 1$.
The axioms of Remark \ref{Rk:frame-recover} 
are direct consequences of Definition \ref{Def:Verj}.
\end{Const}

\begin{Example}
The frames $\u W(R)$, $\u W(B/A)$, $\u\WW(R)$, $\u\WW(B/A)$
arise from 1-frames with Verj\"ungung using this construction.
More precisely, 
for $\u W(R)$ the Verj\"ungung on the ideal $I=I(R)$ is defined
by 
\[
v(a)*v(b)=v(ab),\qquad
\pi(v(a))=pv(a)
\]
for for $a,b\in W(R)$.
For $\u W(B/A)$ with $A=B/J$ the Verj\"ungung on the ideal 
$I=I(B/A)=I(B)\oplus J$ is defined by 
\[
(v(a)+x)*(v(b)+y)=v(ab)+xy,\qquad
\pi(v(a)+x)=pv(a)+x
\]
for $a,b\in W(R)$ and $x,y\in J$.
The Verj\"ungung for $\WW(R)$ and $\WW(B/A)$ is obtained by restriction.
\end{Example}

\section{Graded modules and displays}

Let $S$ be a $\ZZ$-graded ring.
We denote by $\Mod_S^{\gr}$ the category of $\ZZ$-graded $S$-modules
and by $\Hom_S^0(M,N)$ the homomorphisms in this category,
while $\Hom_S(M,N)$ denotes homomorphisms of $S$-modules.
For graded $S$-modules $M$ and $N$ the tensor product $M\otimes_SN$
carries a natural grading, and the category $\Mod_S^{\gr}$ 
has an internal hom given by 
\[
\Hom_S^*(M,N)=\bigoplus_n\Hom_S^n(M,N)
\]
where $\Hom_S^d(M,N)=\Hom_S^0(M,N(d))$ and
$N(d)_n=N_{d+n}$.
We have $\Hom_S^*(M,N)\subseteq\Hom_S(M,N)$,
with equality when $M$ is a finite $S$-module.


\begin{Lemma}
\label{Le:proj-mod-gr}
A graded $S$-module $M$ is projective as an $S$-module iff it is
projective in the category of graded $S$-modules. 
\end{Lemma}

\begin{proof}
If $M$ is projective as a module then also as a graded module
because $\Hom_S^{0}(M,N)$ is a direct summand of $\Hom_S(M,N)$
as an $S_0$-module, thus exact as a functor of $N$.
If $M$ is projective as a graded module, then $M$ is a direct summand
of a direct sum $N=\bigoplus_iS(d_i)$ as a graded module, 
and $N$ is a free $S$-module, thus $M$ is projective as a module.
\end{proof}

For a finite projective graded $S$-module $M$ we define $M^*=\Hom(M,S)$
as a graded $S$-module.
Then the composition map $M^*\otimes_S N\to\Hom_S(M,N)$ is bijective for 
every graded $S$-module $N$.

\subsection{Type of projective graded modules}

%
In the following let $\u S=(S,\sigma,\tau)$ be a frame for $R$, 
thus $R=S_0/tS_1$.
Let 
\begin{equation}
\label{Eq:rho}
\rho:S\to R
\end{equation}
be the ring homomorphism which extends the projection $S_0\to R$
by zero on all $S_n$ with $n\ne 0$. 
Then $\rho$ identifies $S/(S_{<0}+S_{>0})$ with $R$.
If $R$ is considered as a graded ring concentrated in degree zero, 
$\rho$ is a homomorphism of graded rings.
%
We have the following version of Nakayama's lemma.

\begin{Lemma}
\label{Le:NAK}
If $M$ is a finite graded $S$-module,
$M\otimes_{S,\rho}R=0$ implies $M=0$.
\end{Lemma}

\begin{proof}
The ideal $I=tS_1=\Ker(S_0\to R)$ 
is contained in the Jacobson ideal $\Rad(S_0)$
because  $p\in\Rad(S_0)$ by definition,
and the kernel of $S_0/p\to R/p$ is a nil-ideal
by Remark \ref{Rk:S0pRp-nil}.
The kernel of $\rho$ is 
$K=\bigoplus_n K_n$ with $K_0=I$ and $K_n=S_n$ for $n\ne 0$.
Let $x_1,\ldots,x_r$ be homogeneous generators of $M$ 
with $x_i$ of degree $d_i$.
We can write $x_i=\sum a_{ij}x_j$ with $a_{ij}\in K$ of degre $d_i-d_j$.
Then $\det(\delta_{ij}-a_{ij})=1-a$ with $a\in K_0$, 
and $(1-a)x_i=0$.
But $(1-a)$ is a unit of $S_0$ since $I\subseteq\Rad(S_0)$, thus $M=0$.
\end{proof}

\begin{Cor}
\label{Co:NAK-proj}
Let $M$ and $N$ be finite graded $S$-modules where $M$ is projective.
A homomorphism of graded $S$-modules $f:N\to M$ is bijective iff its reduction
$\bar f:N\otimes_{S,\rho}R\to M\otimes_{S,\rho}R$ is bijective. 
\end{Cor}

\begin{proof}
If $\bar f$ is bijecitve, $f$ is surjective by Lemma \ref{Le:NAK}, 
thus $N\cong M\oplus \Ker(f)$ as graded modules
such that $f$ is the projection to the first factor. 
Then $M'=\Ker(f)$ is a finite module with $M'\otimes_{S,\rho}R=0$, thus
$M'=0$ by Lemma \ref{Le:NAK} again.
\end{proof}

\begin{Defn}
\label{Def:std-proj}
A graded $S$-module is called \emph{standard projective} 
if it takes the form $M=L\otimes_{S_0}S$ for a finite projective
graded $S_0$-module $L$.
\end{Defn}

More explicitly, we have $L=\bigoplus_{i} L_i$ 
with finite projective $S_0$-modules $L_i$ which are almost all zero,
and $L\otimes_{S_0}S=\bigoplus_i L_i\otimes_{S_0}S(-i)$.
Every standard projective graded $S$-module is a
finite projective graded $S$-module.

\begin{Lemma}
\label{Le:proj-mod-std}
Assume that all finite projective $R$-modules lift to $S_0$.
Then every finite projective graded $S$-module is standard projective.
\end{Lemma}

\begin{proof}
Let $M$ be a finite projective graded $S$-module.
The graded $R$-module $\bar L=M\otimes_{S,\rho}R$ is finite projective.
Let $L$ be a lift of $\bar L$ to a graded finite projective $S_0$-module
and let $N=L\otimes_{S_0}S$ as a graded $S$-module.
The isomorphism of graded $R$-modules
$M\otimes_{S,\rho}R\cong N\otimes_{S,\rho}R$
lifts to a map of graded $S$-modules $M\to N$,
which is bijective by Corollary \ref{Co:NAK-proj}.
\end{proof}

\begin{Remark}
\label{Rk:mod-std-p-adic}
If the ring $S_0$ is $p$-adic,
the hypothesis of Lemma \ref{Le:proj-mod-std} is satisfied
because the kernel of $S_0/p\to R/p$ is a nil-ideal by 
Remark \ref{Rk:S0pRp-nil},
and thus the homomorphisms $S_0\to S_0/p\to R/p\leftarrow R$
induce bijective maps on the sets of isomorphism classes
of finite projective modules.
\end{Remark}

\begin{Defn}
\label{Def:type-module}
Let $M$ be a finite projective graded $S$-module and let
$\bar L=M\otimes_{S,\rho}R$ as a finite projective graded
$R$-module.
The graded rank of $M$ is the sequence of ranks $(\rk(\bar L_i))_{i\in\ZZ}$,
where $\rk(\bar L_i)$ is a locally constant function on $\Spec R$.
We will also organize this information as follows.
Let $\mu\in\ZZ^n$ with $\mu_1\ge\ldots\ge \mu_n$.
We say that $M$ has type $\mu$ if $\rk(\bar L_i)$ 
is equal to the multiplicity of $i$ in $\mu$.
Set $\deg(\mu)=\sum_i\mu_i$.
\end{Defn}

For example, $M=\bigoplus_i S(-\mu_i)$ has type $\mu$.

\subsection{Predisplays and displays}

Let $\u S=(S,\sigma,\tau)$ be a frame as before.

\begin{Defn}
A \emph{predisplay} $\u M=(M,F)$ over $\u S$ 
consists of a graded $S$-module $M$
and a $\sigma$-linear map $F:M\to M^\tau$.
The predisplay $\u M$ is a \emph{display} 
if $M$ is a finite projective $S$-module
and $F$ is a $\sigma$-linear isomorphism, i.e.\ the corresponding
$S_0$-linear map $M^\sigma\to M^\tau$ is bijective.
The predisplay $\u M$ is called \emph{effective} if $t^n:M_0\to M_{-n}$ 
is bijective for $n\ge 0$. The type of a display is the type of the
underlying graded module as in Definition \ref{Def:type-module}.
\end{Defn}

\begin{Remark}
We have
$M^\tau=\varinjlim(M_0\xrightarrow t
M_{-1}\xrightarrow tM_{-2}\xrightarrow t\ldots)$;
this follows from Remark \ref{Rk:tau}.
A display $\u M$ is effective iff $M_0\to M^{\tau}$ is bijective.
\end{Remark}

%

\begin{Remark}
For a homomorphism of frames $\u S\to\u S'$ there is a base change functor
of predisplays $\u M=(M,F)\mapsto\u M\otimes_{\u S}\u S'=(M\otimes_SS',F\otimes\sigma')$.
This functor is left adjoint to the restriction of scalars functor 
from $\u S'$-predisplays to $\u S$-predisplays. 
The base change preserves displays.
\end{Remark}

\begin{Remark}
Recall that a graded $S$-module $M$ is projective as a module
iff it is projective as a graded module by Lemma \ref{Le:proj-mod-gr}.
A more restrictive variant of the definition of displays would require 
that $\u M$ is a standard projective graded $S$-module as in  
Definition \ref{Def:std-proj}.
This makes no difference when $S_0$ is $p$-adic;
see Lemma \ref{Le:proj-mod-std} and Remark \ref{Rk:mod-std-p-adic}.
\end{Remark}

\begin{Remark}
Let $M$ be a finite projective graded $S$-module.
A homomorphism $M^\sigma\to M^\tau$ 
is bijective iff it is surjective
because $M^\sigma$ and $M^\tau$ 
are finite projective $S_0$-modules of equal rank 
(as locally constant functions on $\Spec S_0$).
Indeed, if $M$ is standard projective, 
i.e.\ $M=L\otimes_{S_0}S$ for a finite projective graded $S_0$-module $L$,
then $M^\tau=L$ and $M^\sigma=L^{\sigma_0}$, 
and the equality of ranks holds 
since $\sigma_0$ is a Frobenius lift and $p\in\Rad(S_0)$.
In general it suffices to show that 
$M^\sigma\otimes R/p$ and $M^\tau\otimes R/p$ have equal rank.
Using Remark \ref{Rk:frame-W1-final} we can assume that $R=R/p$ 
and that $S=\uW 1(R)$ is the associated zip frame. 
In that case $M$ is standard projective.
\end{Remark}


\begin{Remark}
\label{Rk:disp-Frob-mod}
If $\u M$ is an effective predisplay, 
we can identify $M_0=M^\tau$,
and $F_0$ defines a $\sigma_0$-linear endomorphism of $M^\tau$.
If $\u M$ is an arbitrary display, 
for sufficiently large $d$ we can identify $M_{-d}=M^\tau$,
and $(M^\tau,F_{-d})$ is a finite projective Frobenius module over $S_0$,
with $F_{-d-1}=pF_{-d}$.
\end{Remark}


\begin{Remark}
Predisplays and displays are related to the $\varphi$-gauges of \cite{Fontaine-Jannsen:Frobenius-gauges}.
We refer to \S\ref{Subse:DisplayGauge} for details.
\end{Remark}

\subsection{Normal decomposition and standard data}

Let $L=\bigoplus_iL_i$ be a finite projective graded $S_0$-module
and let
\begin{equation}
\label{Eq:M-std}
M=L\otimes_{S_0}S=\bigoplus_i L_i\otimes_{S_0}S(-i)
\end{equation}
be the associated standard projective graded $S$-module,
thus $L_i\otimes_{S_0}S_j$ lies in degree $i+j$.
We consider $L$ as a submodule of $M$ by
$L\to M$, $x\mapsto x\otimes 1$.
The composition $L\to M\to M^\tau$, $x\mapsto x\otimes 1\mapsto x\otimes 1$  is bijective, and
\[
\Hom_\sigma(M,M^\tau)\xrightarrow\sim\Hom_{\sigma_0}(L,M^\tau)
\cong\Hom_{\sigma_0}(L,L)
\]
by sending $F$ to the restriction $\Phi=F|_L$.
The pair $(M,F)$ is a display iff $\Phi$ is a $\sigma_0$-linear isomorphism.

\begin{Defn}
A standard datum is a pair $(L,\Phi)$ where $L$ is a finite projective
graded $S_0$-module and $\Phi:L\to L$ is a $\sigma_0$-linear isomorphism.
A normal decomposition of a display $\u M=(M,F)$ 
is a finite projective graded $S_0$-module $L\subseteq M$ 
such that $M=L\otimes_{S_0}S$ as a graded $S$-module.
\end{Defn}

A normal decomposition $L$ of a display $\u M$ gives a standard
datum $(L,\Phi)$ with $\Phi=F|_L$, using the identification
$M^{\tau}\cong L$. Conversely, a standard datum
$(L,\Phi)$ gives a display $(M,F)$ defined by $M=L\otimes_{S_0}S$
and $F(x\otimes s)=\sigma(s)\Phi(x)$ for $x\in L$ and $s\in S$.
If each finite projective $R$-module lifts to $S_0$,
then every display has a normal decomposition 
and thus comes from a standard datum;
see Lemma \ref{Le:proj-mod-std}.

\subsection{Change of basis}
\label{Se:disp-change-basis}

Let $M$ be a finite projective graded $S$-module. 
We consider the set $X$ of $\sigma$-linear isomorphisms 
$F:M\to M^\tau$ 
and the group $G=\Aut_S^{0}(M)$ 
of automorphisms of the graded $S$-module $M$. 
Then $G$ acts on $X$ by 
\begin{equation}
\label{Eq:conj-1}
X\times G\to X,\qquad
(F,g)\mapsto (g^\tau)^{-1}\circ F\circ g
\end{equation}
Equivalently, $g$ acts on the linear isomorphism
$F':M^\sigma\to M^\tau$ corresponding to $F$
by $F'\mapsto(g^\tau)^{-1}\circ F'\circ g^\sigma$.
The groupoid of displays with underlying graded $S$-module isomorphic to $M$
is the quotient groupoid $[X/G]$.

The action \eqref{Eq:conj-1} can be made explicit as follows.
For $\mu=(\mu_1,\ldots,\mu_n)\in\ZZ^n$ with $\mu_1\ge\ldots\ge\mu_n$
we consider the graded $S_0$-module $L=\bigoplus_iS_0(-\mu_i)$ and the
associated graded $S$-module
\[
M=L\otimes_{S_0}S=\bigoplus_iS(-\mu_i).
\]
There are natural identifications $M^\tau=L=S_0^n$ and $M^\sigma=L^{\sigma_0}=S_0^n$ as $S_0$-modules,
which gives a bijective map
\[
X=\Isom_{S_0}(M^\sigma,M^\tau)
\cong\GL_n(S_0).
\]
The endomorphisms of $M$ as a graded module are described by matrices:
\begin{equation}
\label{Eq:End-gr-M}
\End_S^{0}(M)\cong
\{A=(a_{ij})\in M_n(S)\mid 
a_{ij}\in S_{\mu_j-\mu_i}\}.
\end{equation}
Here the entries $a_{ij}$ with $\mu_i=\mu_j$ and thus $a_{ij}\in S_0$
form a sequence of square matrices $A_1,\ldots, A_r$, 
where $r$ is the number of different entries in $\mu$,
\[
A=
\left(
\begin{matrix}
A_1 && * \\
& \ddots \\
* && A_r
\end{matrix}
\right).
\]

As an example, for $\mu=(3,2,2,0)$ we have
\[
\End_S^{0}\bigl(S(-3)\oplus S(-2)^2\oplus S(0)\bigr)=
\left(
\begin{matrix}
S_0 & S_{-1} & S_{-1} & S_{-3} \\
S_1 & S_0 & S_0 & S_{-2} \\
S_1 & S_0 & S_0 & S_{-2} \\
S_3 & S_2 & S_2 & S_0
\end{matrix}
\right)
\]
with $r=3$ and square matrices $A_1$, $A_2$, $A_3$ of order $1$, $2$, $1$.

\begin{Lemma}
\label{Le:AutM}
The group $G=\Aut_S^{0}(M)\subseteq\End_S^{0}(M)$ consists of all 
matrices $A$ as in \eqref{Eq:End-gr-M}
such that the diagonal blocks $A_1,\ldots,A_r$ are invertible.
\end{Lemma}

\begin{proof}
By Lemma \ref{Le:NAK}, a matrix $A$ as in \eqref{Eq:End-gr-M}
defines a surjective endomorphism
of $S^n$ iff its image $\rho(A)$ defines a surjective endomorphism of $R^n$.
In both cases, surjective is equivalent to bijective.
Moreover $\rho(A)$ is the block diagonal matrix with diagonal blocks 
$\rho(A_1),\ldots,\rho(A_r)$.
Thus $A$ is invertible iff $\rho(A)$ is invertible iff all $\rho(A_i)$ are invertible iff all $A_i$ are invertible; 
recall that the kernel of $S_0\to R$ lies in $\Rad(S_0)$
by the proof of Lemma \ref{Le:NAK}.
\end{proof}

The ring homomorphisms $\sigma,\tau:S\to S_0$ 
induce group homomorphisms $\sigma,\tau:G\to\GL_n(S_0)$,
and in terms of matrices the action \eqref{Eq:conj-1}
is given by 
\begin{equation}
\label{Eq:Conj}
X\times G\to X,\qquad
(B,A)\mapsto \tau(A)^{-1}\cdot B\cdot \sigma(A).
\end{equation}
Here $\sigma$ acts on $A=(a_{ij})$ by $a_{ij}\mapsto\sigma_{\mu_j-\mu_i}(a_{ij})$.

\subsection{Exact and tensor structure}

The category of predisplays has an obivous tensor product,
\[
(M,F)\otimes(M',F')=(M\otimes_SM,F\otimes F'),
\]
using the identification 
$(M\otimes_SN)^\tau=M^\tau\otimes_{S_0}N^\tau$
and similarly for $\sigma$.
The tensor product of predisplays preserves displays because 
the tensor product of graded $S$-modules preserves finite projective modules. 
The display $\u S=(S,\sigma)$ is the unit object for the tensor 
product of (pre)displays.
For $d\in\ZZ$ let $\u U(d)=(S(-d),\sigma)$ as a display. 
Then $\u U(d)$ is invertible with inverse $\u U(-d)$.
The type of $\u U(d)$ is $(d)\in\ZZ^1$.

A sequence of displays $0\to \u M'\to\u M\to\u M''\to 0$ is called exact if it is exact in
the abelian category of predisplays, i.e.\ if it is exact as a sequence of $S$-modules.
This makes displays into an exact category, 
which is Karoubian (idempotent complete).
The tensor product of displays preserves exactness.

We define a bilinear form of predisplays $\beta:\u M\times\u N\to\u P$ 
to be a bilinear map of the underlying graded $S$-modules 
$\beta:M\times N\to P$
such that $F(\beta(x,y))=\beta^\tau(F(x),F(y))$
where $\beta^\tau:M^\tau\times N^\tau\to P^\tau$
is the induced bilinear map of $S_0$-modules.
A bilinear form $\beta:\u M\times\u N\to\u P$ 
corresponds to a homomorphism of predisplays
$\u M\otimes \u N\to \u P$.

The dual of a display $\u M=(M,F)$ is the display $\u M^*=(M^*,F^*)$
where $M^*=\Hom_S(M,S)$ as a graded $S$-module, and the isomorphism 
$(M^*)^\sigma\to(M^*)^\tau$ corresponding to $F^*$ is the dual of the
inverse of the isomorphism $M^\sigma\to M^\tau$ corresponding to $F$.
The canonical bilinear form $\gamma:M^*\times M\to S$ is a bilinear form
of displays $\u M^*\times\u M\to\u S$. 

For a display $\u M$ and a predisplay $\u P$ we define 
$\uHom(\u M,\u P)=\u P\otimes\u M^*$; 
this is a display if $\u P$ is a display.
For another predisplay $\u N$ the homomorphism
\[
\Hom(\u N,\u\Hom(\u M,\u P))\to\Hom(\u N\otimes\u M,\u P)
\]
defined by 
$\lambda\mapsto({\id}_{\u P}\otimes\gamma)\circ(\lambda\otimes{\id}_{\u M})$
is bijective.

A bilinear form of displays $\u N\times\u M\to \u S$ is called perfect
if the underlying bilinear form of $S$-modules is perfect,
or equivalently if the corresponding homomorphism of displays
$\u N\to\u M^*$ is an isomorphism.
More generally, 
if $\u P$ is a display of rank one, a bilinear form of displays
$\u N\times\u M\to \u P$ will be called perfect if the corresponding 
homomorphism $\u N\to\u M^*\otimes P$ is an isomorphism.

\subsection{Examples}

Let us record what the general definition of displays gives for specific frames.

\begin{Example}
If $k$ is a perfect field,
displays over the Witt frame $\u W(k)$ are
equivalent to pairs $(N,\Phi)$ where $N$ is a finite free $W(k)$-module
and $\Phi$ is a $\sigma_0$-linear isomorphism of $N[1/p]$.
\end{Example}

Using the relation between displays and $\varphi$-gauges 
as explained in \S\ref{Subse:DisplayGauge}, in particular
Example \ref{Ex:Witt-gauge}, this is the statement of
\cite[Cor.\ 2.2.10]{Fontaine-Jannsen:Frobenius-gauges}.
We include a proof for completeness.

\begin{proof}
The graded ring $S$ is the ring of all $\sum a_nt^{-n}\in W(k)[t,t^{-1}]$
such that $p^n\mid a_n$ for $n\ge 1$, with $\deg(t)=-1$.
Using the elementary divisors theorem 
one verifies that finite projective graded $S$-modules $M$ 
are equivalent to pairs $(N,N')$ 
where $N$ is a finite free $W(k)$-module
and $N'$ is a lattice in $N[1/p]$
by the following functors.
To $M$ one associates the $S_0$-module
$N=M^\tau$ with $N'=\sum_ip^{-i}\Image(M_i\to M^\tau)$,
and to $(N,N')$ one associates the graded $S$-module $M$
with $M_i=N\cap p^iN'$ 
such that $t:M_{i+1}\to M_i$ is the inclusion.

If $M$ corresponds to $(N,N')$ 
we consider $M_i$ as a submodule of $N=M^\tau$.
Then a $\sigma$-linear map $F:M\to M^\tau$
is a system of $\sigma_0$-linear maps $F_i:M_i\to N$ with $pF_{i+1}=F_i$,
which corresponds to a $\sigma_0$-linear map $\Phi:N\to N[1/p]$
with $\Phi(N')\subseteq N$,
defined by $\Phi(x)=F_0(x)$ for $x\in M_0$.
The pair $(M,F)$ is a display iff $\Phi(N')=N$.
Thus $N'$ is determined by $\Phi$, 
and $N'$ exists iff $\Phi$ is bijective on $N[1/p]$.
%
\end{proof}


\begin{Example}
\label{Ex:disp-tautological}
Let $\u S$ be a tautological frame 
as in Example \ref{Ex:frame-tautological},
thus $S=S_0[t]$. 
A display over $\u S$ is equivalent to 
a finite projective $S_0$-module $N$ 
with an exhaustive and separating decreasing filtration 
by direct summands $(N_i)_{i\in\ZZ}$, 
together with $\sigma_0$-linear maps $F_i:N_i\to N$ such that $pF_{i+1}=F_i$
and such that $N$ is generated by the union of all $F_i(N_i)$.
\end{Example}

\begin{Example}
For a 1-frame $\SSS$ with Verj\"ungung $(\nu,\pi)$, 
Langer and Zink \cite{Langer-Zink:K3} define predisplays, displays,
and extended displays.
Let $\u S$ be the frame that extends $\SSS$ as in Construction \ref{Const:1fr2fr},
and let $\u S^{\min}$ be the minimal subframe of $\u S$ as in Remark \ref{Rk:Smin}.
Then an effective predisplay over $\u S^{\min}$ 
is the same as a predisplay over $\SSS$
that satisfies the condition alpha 
of equation (7) of \cite{Langer-Zink:K3}.
Effective displays over $\u S^{\min}$ 
which admit a normal decomposition
are equivalent to displays over $\SSS$,
while effective displays over $\u S$ 
which admit a normal decomposition are equivalant to
extended displays over $\SSS$.
These equivalences preserves the tensor product of displays;
to prove this one verifies that both tensor products represent the
same notion of bilinear maps of displays.
Specific cases are:
\begin{enumerate}
\item
For $\u S=\u W(R)$ or $\u S=\u\WW(R)$ we have $\u S=\u S^{\min}$,
so effective displays over $\u S$ are equivalent 
to displays over $\SSS$. In the case of $\u W(R)$ these
coincide with the higher displays of \cite{Langer-Zink:DRW-Disp}. 
\item
For $\u S=\u W(B/A)$ or $\u S=\u\WW(B/A)$ we have $\u S\ne\u S^{\min}$ in general.
In this case, effective displays over $\u S^{\min}$ 
are equivalent to displays over $\SSS$, 
while effective displays over $\u S$ are equivalent to 
extended displays over $\SSS$.
\end{enumerate}
\end{Example}

\begin{Example}
[Displays over $\uW 1(R)$ and $F$-zips]
\label{Ex:disp-W1-zip}
Let $\u S=\uW 1(R)$ for an $\FF_p$-algebra $R$; 
see Example \ref{Ex:frame-W1} and Remark \ref{Rk:frame-W1-final}.
Displays over $\uW 1(R)$ are equivalent to 
the $F$-zips over $R$ of \cite{Moonen-Wedhorn:Discrete} 
as a tensor category. 
\end{Example}

We recall that an $F$-zip over $R$ is a finite projective $R$-module
$\MMM$ with two filtrations $C^*$ (ascending) and $D_*$ (decending)
by direct summands and an isomorphism $(\gr^*_C)^{(p)}\cong\gr_*^D$.
The above equivalence is a variant of a result of \cite{Schnellinger},
who proves that the modified $F$-zips of \cite{Wedhorn:De-Rham} are
equivalent to $\varphi$-$R$-gauges in the sense of 
\cite{Fontaine-Jannsen:Frobenius-gauges}; 
see \S\ref{Subse:DisplayGauge} and in particular Example \ref{Ex:zip-gauge}.

\begin{proof}
To prove the equivalence between displays and $F$-zips,
let $\CCC_R$ be the category of finite projective graded $S$-modules,
moreover let $\DDD_R$ be the category of collections
\[
(\MMM,C^*,\NNN,D_*,\alpha)
\]
where $\MMM$ and $\NNN$ are finite projective $R$-modules,
$C^*$ is an ascending filtration of $\MMM$ and $D_*$ a decending
filtration of $\NNN$ by direct summands, 
and $\alpha:(\gr^*_C)^{(p)}\cong\gr_*^D$ is
an isomorphism of graded $R$-modules. 
The categories $\CCC_R$ and $\DDD_R$ are equivalent by the following functors.

We define a functor
\[
\CCC_R\to\DDD_R,\qquad M\mapsto (\MMM,C^*,\NNN,D_*,\alpha)
\]
by $\MMM=M^\tau$ and $\NNN=M^\sigma$ where $C^i\subseteq\MMM$ 
is the image of $M_i$ and $D_i\subseteq\NNN$ is generated by the
image of $M_i$. 
The natural surjective maps $M_i^{(p)}\to(\gr^i_C)^{(p)}$
and $M_i^{(p)}\to\gr_i^D$ have the same kernel and therefore
induce the desired isomorphism $\alpha_i:(\gr^i_C)^{(p)}\cong\gr_i^D$.
Indeed, the equality of kernels and the axioms of an object of $\DDD_R$
are easily verified
when $M=L\otimes_RS(-d)$ for a finite projective $R$-module,
and the general case follows because every $M\in\CCC_R$ is a finite direct sum of
such modules. 

Conversely, we define a functor 
\[
\DDD_R\to\CCC_R,\qquad(\MMM,C^*,\NNN,D_*,\alpha)\mapsto M
\]
by $M_i=C^i\times_{\gr_i^D}D_i=
\{(x,y)\in C^i\times D_i\mid\alpha_i(\bar x\otimes 1)=\bar y\text{ in }\gr_i^D\}$.
This gives a graded $S_0$-module $M=\bigoplus_iM_i$. 
The $S$-module structure is given as follows. Let $n\ge 1$.
For $a\in S_{-n}=R$ define $a:M_i\to M_{i-n}$ by
the multiplication $a:C^i\to C^{i-n}$ and by zero on $D_i$.
For $a\in S_n=R$ define $a:M_i\to M_{i+n}$ by 
zero on $C^i$ and by the multiplication $a:D_i\to D_{i+n}$.
The fact that this construction is well-defined is easily verified
if $\gr^*_C$ is concentrated in one degree, and the general case follows
because every object of $\DDD_R$ is a finite direct sum of such objects.
Similarly one verifies that the functors are mutually inverse.

Now, a display over $\uW 1(R)$ consists of a module $M\in\CCC_R$ 
together with an isomorphism
$F:M^\sigma\cong M^\tau$, which corresponds to an isomorphism 
$u:\NNN\cong\MMM$ for the associated object of $\DDD_R$.
Then $(\MMM,C^*,u(D_*),\gr_*(u)\circ\alpha)$ is an $F$-zip over $R$.
\end{proof}

\subsection{Displays and Frobenius gauges}
\label{Subse:DisplayGauge}

Displays are related with the theory of $\varphi$-gauges of \cite{Fontaine-Jannsen:Frobenius-gauges} as follows. 

\begin{Point}
\label{Pt:frame-fgauge}
Let $\u S$ be a frame with an element $\fgauge\in S_1$ such that $\sigma_1(\fgauge)=1$, moreover let $\vgauge=t\in S_{-1}$. 
(For an arbitrary frame, 
$\fgauge$ need not exist, for example the tautological frame 
of Example \ref{Ex:frame-tautological} has $S_1=0$.)
As in \cite{Fontaine-Jannsen:Frobenius-gauges} 
we write
\[
S_{-\infty}=S/(1-\vgauge)\qquad\text{and}\qquad S_{+\infty}=S/(1-\fgauge).
\] 
The homomorphisms $\sigma,\tau:S\to S_0$ induce an isomorphism 
$\underline\tau:S_{-\infty}\xrightarrow\sim S_0$ and a homomorphism 
$\underline\sigma:S_{+\infty}\to S_0$. 
Let $\varphi=\underline\tau^{-1}\circ\underline\sigma$.
Then $(S,\varphi)$ is a $\varphi$-ring in the sense of 
\cite[\S 1.4]{Fontaine-Jannsen:Frobenius-gauges}, and
predisplays over $\u S$ correspond to $\varphi$-$S$-modules,
moreover a predisplay $(M,F)$ over $\u S$
corresponds to a $\varphi$-$S$-gauge
iff $F:M^\sigma\to M^\tau$ is bijective.
In particular, displays over $\u S$ correspond to $\varphi$-$S$-gauges 
whose underlying graded $S$-module is finite projective.
\end{Point}

\begin{Point}
\label{Pt:gauge-S0}
Conversely, let $S_0$ be a ring with a Frobenius lift $\sigma_0$ such that $p\in\Rad(S_0)$, and let $S=S_0[\fgauge,\vgauge]/(\fgauge\vgauge-p)$ as a graded $S_0$-algebra with $\deg(\fgauge)=1$ and $\deg(\vgauge)=-1$.
Then we can identity $S_{+\infty}=S_0=S_{-\infty}$, and $(S,\sigma_0)$ is a $\varphi$-ring.
Moveover let $\tau:S\to S_0$ extend the identity of $S_0$ by $\tau(\vgauge)=1$ and $\tau(\fgauge)=p$,
and let $\sigma:S\to S_0$ extend $\sigma_0$ by $\sigma(\vgauge)=p$ and $\sigma(\fgauge)=1$. 
Then $(S,\sigma,\tau)$ is a frame that gives the $\varphi$-ring $(S,\sigma_0)$ by the above construction,
so $\varphi$-$S$-modules correspond to predisplays over this frame etc.
\end{Point}

\begin{Point}
If $\u S$ is a frame with $\fgauge\in S_1$ as in \ref{Pt:frame-fgauge}, the construction of \ref{Pt:gauge-S0} applied to the pair $(S_0,\sigma_0)$ gives a frame $\u S'$ with $S'=S_0[\fgauge,\vgauge]/(\fgauge\vgauge-p)$. There is a  frame homomorphism
\[
\psi_-:\u S\to\u S'
\]
defined by $\psi_-(a_0t^n)=\sigma_0(a_0)\vgauge^n$ for $a_0\in S_0$ and $n\ge 0$, and $\psi_-(a_n)=\sigma_n(a)\fgauge^n$ for $a_n\in S_n$ and $n\ge 1$.
If $\fgauge\vgauge=1$ in $S$, there is also a frame homomorphism in the opposite direction
\[
\psi_+:\u S'\to\u S
\]
which extends the identity of $S_0$ by $\vgauge\mapsto t$ and $\fgauge\mapsto\fgauge$.
The composition $\psi_+\circ\psi_-$ extends $\sigma_0$ on $S_0$ by
$t\mapsto t$ and $a_n\mapsto\fgauge^n\sigma_n(a_n)$ for $a_n\in S_n$ with $n\ge 1$.
\end{Point}

\begin{Example}
\label{Ex:Witt-gauge}
Let $\u S=\u W(R)$ for a $p$-adic ring $R$. 
Then $\sigma_n:S_n\to S_0$ is bijective for $n\ge 1$,
so there is a unique $\fgauge\in S_1$ with $\sigma_1(\fgauge)=1$, namely $\fgauge=v(1)$,
moreover $\underline\sigma$ is bijective, 
which means that $(S,\varphi)$ is a perfect $\varphi$-ring
in the sense of \cite[\S 1.4]{Fontaine-Jannsen:Frobenius-gauges}.
We have $\fgauge\vgauge=p$ iff $R$ is an $\FF_p$-algebra.
In this case, $\psi_+$ is an isomorphism iff $R$ is perfect 
iff $\psi_-$ is an isomorphism, and consequently displays over 
a perfect ring $R$
are equivalent to $\varphi$-$W(R)$-gauges whose underlying graded module
is finite projective; cf.\ \cite{Wid}.
If $R$ is a general $\FF_p$-algebra, the restriction of scalars by
$\psi_+$ defines a functor from displays over $R$ to $\varphi$-$W(R)$-modules, which happens to be fully faithful because the endomorphisms of $S$ as an $S$-module and as a $W(R)[\fgauge,\vgauge]$-module coincide.
\end{Example}

\begin{Example}
\label{Ex:zip-gauge}
Let $\u S=\uW 1(R)$ for an $\FF_p$-algebra $R$.
Again there is a unique $\fgauge\in S_1$ with $\sigma_1(\fgauge)=1$,
namely $\fgauge=u$ under the isomorphism \eqref{Eq:W1}.
If we identify
\[
S=\{(f,g)\in R[\vgauge]\times R[\fgauge]\mid f(0)^p=g(0)\},
\]
\[
S'=\{(f,g)\in R[\vgauge]\times R[\fgauge]\mid f(0)=g(0)\},
\]
the homomorphism $\psi_-:S\to S'$ applies the Frobenius to the coefficients of $f$ without changing $g$, and the homomorphism $\phi_+:S'\to S$ applies the Frobenius to the coefficients of $g$ without changing $f$.
By Example \ref{Ex:disp-W1-zip},
displays over $\u S$ are equivalent to the $F$-zips of \cite{Moonen-Wedhorn:Discrete},
which are also used in \cite{Pink-Wedhorn-Ziegler:F-zips},
while displays over $\u S'$ are equivalent to the modified $F$-zips of
\cite{Wedhorn:De-Rham} by \cite{Schnellinger}.
The base change under $\psi_-$ corresponds to the functor
from $F$-zips to modified $F$-zips  defined by
$(M,C^*,D_*,\varphi)\mapsto(M,(C^*)^{(p)},D_*,\varphi)$,
while the base change under $\psi_+$ correponds to the functor
in the opposite direction 
$(M,C^*,D_*,\varphi)\mapsto(M^{(p)},C^*,(D_*)^{(p)},\varphi^{(p)})$.
\end{Example}

\section{Descent}

The frames $\u W(R)$ and $\u\WW(R)$ are functors of $R$,
and naturally the associated graded modules and displays 
satisfy descent with respect to some topology on $\Spec R$.
The details differ in the two cases.
For $\u W(R)$ we have fpqc descent 
using that the ring $W(R)$ is the limit of $W_m(R)$;
this is a variant of the Witt vector descent of \cite{Zink:Display}.
In the case of $\u\WW(R)$ we get etale descent using the
observation that for an etale homomorphism of admissible
rings $R\to R'$ (necessarily finite etale) 
the homomorphism $\WW(R)\to\WW(R')$ is finite etale again.
This obervation can be formalized:
If $\u S$ is a frame with $p$-adic components $S_n$
for a ring $R$ in which $p$ is nilpotent,
an etale homomorphism $R\to R'$ can be extended 
to a frame homomorphism $\u S\to\u S'$ 
so that $\u S'$ is a functor of $R'$.
We will verify etale descent in this situation.


\subsection{Sheaf properties of the Witt frame}

The frames $\u W(R)$ form an fpqc sheaf with respect to $R$ because
this holds for the ring $W(R)$, which is bijective to $R^\NN$ as a set.
If $B\to A$ is a PD thickening of rings in which $p$ is nilpotent,
each etale homomorphism $A\to A'$ 
lifts to a unique etale homomorphism $B\to B'$,
and the assignment $A'\mapsto\u W(B'/A')$ is an etale sheaf of frames.

\subsection{Etale localization of frames}
\label{Se:etale-loc}



%

\begin{Defn}
\label{Def:p-adic-frame}
A frame $\u S$ is called $p$-adic if each $S_n$ is a $p$-adic 
abelian group and $p$ is nilpotent in the ring $R=S_0/tS_1$.
\end{Defn}

\begin{Lemma}
\label{Le:p-adic-frame}
$p$-adic frames $\u S$ are equivalent to systems of 
frames $(\u S^{(n)})_{n\in\NN}$ with isomorphisms 
$\u S^{(n+1)}/p^n\cong\u S^{(n)}$
such that the corresponding sequence of rings 
$R^{(n)}=S^{(n)}_0/tS^{(n)}_1$  is stationary, 
by sending $\u S$ to $\u S^{(n)}=\u S/p^n$.
\end{Lemma}

We note that $R^{(n)}=R^{(n+1)}/p^n$ for all $n$, 
so the sequence $R^{(n)}$ is constant for $n\ge m$
iff $p^mR^{(m+1)}=0$.

\begin{proof}
Since $p$-adic abelian groups are equivalent to systems
of abelian groups $(A^{(n)})_{n\in\NN}$ with $A^{(n+1)}/p^n=A^n$,
frames $\u S$ with $p$-adic components $S_n$ are equivalent
to systems $(\u S^{(n)})_{n\in\NN}$ as in the lemma 
with no condition on $R^{(n)}$.
In this situation let $R=S_0/tS_1$, thus $R^{(n)}=R/p^n$. 
In general $R$ need not be $p$-adic.
Assume that $R/p^m=R/p^{m+1}$ for some $m$.
Let $N\subseteq S_1$ be the kernel of $t:S_1\to S_0/p^m$.
Then $N$ is $p$-adic since $S_1$ is $p$-adic, 
moreover $N\to p^mS_0/p^{m+1}S_0$ is surjective. 
Hence $N\to p^mS_0$ is surjective, and thus $R=R_m$.
\end{proof}

\begin{Lemma}
\label{Le:etale-bc-frame}
Let $\u S$ be a $p$-adic frame for $R$ 
as in Definition \ref{Def:p-adic-frame}
and let $R\to R'$ be an etale ring homomorphism.
There is a unique $p$-adic frame $\u S'$ for $R'$
with a homomorphism $\u S\to\u S'$
such that $S_0/p^i\to S_0'/p^i$ is etale and 
$R\otimes_{S_0}S_0'=R'$ and 
$S/p^i\otimes_{S_0}S_0'=S'/p^i$.
\end{Lemma}

We call $\u S'$ the base change of $\u S$ under $R\to R'$
and write $\u\bS(R')=\u S'$. 
Then $\u\bS$ is a presheaf of frames on the category of
affine etale $R$-schemes.

\begin{proof}
By Lemma \ref{Le:p-adic-frame} we may assume that $p^mS=0$. 
Then the kernel of $S_0\to R$ is a nil-ideal, 
so the etale $R$-algebra $R'$ lifts to a unique etale $S_0$-algebra $S_0'$, 
and we have to set $S'=S_0'\otimes_{S_0}S$ as a graded ring
with $\tau={\id}\otimes\tau$.
Since $S_0\to S_0'$ is etale, there is a unique homomorphism
$\sigma_0$ in the middle such that the following diagram commutes.
\[
\xymatrix@M+0.2em{
S_0 \ar[r] \ar[d]_{\sigma_0} &  
S_0' \ar[r] \ar@{-->}[d]^{\sigma_0} & 
S_0'/p \ar[d]^{\Frob} \\
S_0 \ar[r] & S_0' \ar[r] & S_0'/p  
}
\]
This also determines $\sigma=\sigma_0\otimes\sigma$ on $S'=S'_0\otimes_{S_0}S$.
\end{proof}

\begin{Lemma}
\label{Le:S-sheaf}
The presheaf\/ $\u\bS$ is an etale sheaf.
\end{Lemma}

\begin{proof}
This is immediate from the construction, using that quasi-coherent modules are
etale sheaves and that limits of sheaves are sheaves.
\end{proof}

\begin{Lemma}
\label{Le:frame-WW-bc}
Let $A$ and $A'$ be admissible local Artin rings rings 
as in Example \ref{Ex:frame-WW}
and let $A\to A'$ be an etale ring homomorphism. 
Then $\u\WW(A')$ is the base change of\/ $\u\WW(A)$ under $A\to A'$.
If $B\to A$ is a PD thickening of admissible local Artin rings
with nilpotent divided powers
and if $B\to B'$ is the unique etale homomorphism that lifts $A\to A'$,
then $\u\WW(B'/A')$ 
is the base change of\/ $\u\WW(B/A)$ under $A\to A'$.
\end{Lemma}

\begin{proof}
Let $\Fm\subset A$ and $\Fm'\subset A'$ be the maximal ideals,
$k=A/\Fm$, and $k'=A'/\Fm'$.
We write $\u S=\u\WW(A)$ and $\u S'=\u\WW(A')$.
Let $i$ be sufficiently large such that $p^i\hat W(\Fm')=0$.
Then $S_0/p^i=\hat W(\Fm)\oplus W_i(k)$ and 
$S_n/p^i=\hat I(\Fm)\oplus I_{i+1}(k)$ for $n\ge 1$,
moreover we have a sequence of ring homomorphisms
\[
W_i(k)\to S_0/p^i\to W_i(k)\to k
\]
where the composition of the first two maps is the identity.
All this also holds for $A'$ instead of $A$ with $'$ in appropriate places. 
We have to show that the homomorphism $S_0/p^i\to S'_0/p^i$ 
is etale and lifts $k\to k'$.
This is true for the homomorphism $W_i(k)\to W_i(k')$, 
hence we have to show that the natural map
\begin{equation}
\label{Eq:S0WW}
S_0/p^i\otimes_{W_i(k)}W_i(k')\to S_0'/p^i
\end{equation}
is bijective, or equivalently that
\begin{equation*}
\hat W(\Fm)\otimes_{W_i(k)}W_i(k')\to\hat W(\Fm')
\end{equation*}
is bijective. 
This holds if for each $N=\Fm^r/\Fm^{r+1}$ and $N'=\Fm'^r/\Fm'^{r+1}$
the homomorphism 
\[
\hat W(N)\otimes_kk'\to\hat W(N')
\] 
is bijective,
which is clear since $N'=N\otimes_kk'$ and $\hat W(N)=N^{(\NN)}$.
It also follows that $S_0\to S_0'$ 
is a finite etale lift of $k\to k'$.
Moreover we have to show that 
\begin{equation}
\label{Eq:SnSS}
S_n/p^i\otimes_{S_0}S_0' \to S'_n/p^i
\end{equation}
is biective for all $n\in\ZZ$,
or equivalently that
\begin{equation}
\label{Eq:SnWW}
S_n/p^i\otimes_{W_i(k)}W_i(k') \to S'_n/p^i
\end{equation}
is bijective.
This is proved as \eqref{Eq:S0WW}, 
using that $S_n\cong S_0$ for $n\le 0$
and $S_n/p^i\cong\hat I(\Fm)\oplus W_i(k)$ for $n\ge 1$.

Now let $B\to A=B/J$ and $B'\to A'=B'/J'$ 
be as in the second part of the lemma,
and let $\u S=\u\WW(B/A)$ and $\u S'=\u\WW(B'/A')$.
Since $J'=JB'$ is a PD ideal with nilpotent divided powers,
the frame $\u\WW(B'/A')$ is defined.
Let $i$ be sufficiently large such that $p^i\hat W(\Fm_{B'})=0$.
Again we have to show that 
\eqref{Eq:SnWW} is bijective for all $n\in\ZZ$.
In degrees $n\le 0$ the frame $\u\WW(B/A)$ coincides with $\u\WW(B)$,
which was treated before.
For $n>0$ we have $S_n/p^i\cong\hat I(\Fm_B)\oplus J\oplus W_i(k)$,
and here $J'=J\otimes_BB'=J\otimes_{W_i(k)}W_i(k')$ as required.
\end{proof}

\begin{Example}
Let $A\to A'$ be an etale homomorphism of $\FF_p$-algebras and let $m\ge 1$.
Then $\uW m(A')$ is the base change of $\u W(A)$ under $A\to A'$.
If $B\to A=B/J$ is a PD thickening of $\FF_p$-algebras and $B\to B'$
is the etale homomorphism that lifts $A\to A'$, then $\uW m(B'/A')$ is the base change of $\uW m(B/A)$ under $A\to A'$.
The proof is easy, using \cite[Prop.~A.8]{Langer-Zink:DRW}.
\end{Example}

\begin{Example}
\label{Ex:frame-W-bc}
If $A\to A'$ is an etale homomorphism of rings in which $p$ is nilpotent,
$\u W(A')$ does not in general coincide with the base change of $\u W(A)$
under $A\to A'$ because the limit topology of $W(A)=\varprojlim_m W_m(A)$ 
differs from the $p$-adic topology. 
\end{Example}

\subsection{Descent of graded modules}

\begin{Lemma}
\label{Le:descent-grad-mod-S}
Let $\u S$ a $p$-adic frame for $R$ 
(Definition \ref{Def:p-adic-frame}).
The categories of finite projective graded $\bS(R')$-modules 
for varying etale $R$-algebras $R'$ form an etale stack.
\end{Lemma}

%

\begin{proof}
Since for finite projective graded $S$-modules $M$ and $N$ the
graded $S$-module $\Hom_S(M,N)$ is finite projective again,
the homomorphisms form an etale sheaf by Lemma \ref{Le:S-sheaf}.
Let $R\to R'$ be faithfully flat etale and $R''=R'\otimes_RR'$.
We have to show that a finite projective graded module $M'$ over 
$S'=\bS(R')$ with a descent datum over $S''=\bS(R'')$ 
descends to a finite projective graded $S$-module.

Assume first that $p^mS=0$. 
Then $S_0\to S_0'$ is faithfully flat etale, 
and the graded $S_0'$-module $M'$ descends to a graded $S_0$-module $M$.
Since $S'=S\otimes_{S_0}S_0'$, the multiplication map $S'\times M'\to M'$
descends to $S\times M\to M$, and $M$ is a graded $S$-module such that $M'=M\otimes_SS'$.
Since $S\to S'$ is faithfully flat,
the $S$-module $M$ is finite projective,
and $M$ is projective as a graded $S$-module 
by Lemma \ref{Le:proj-mod-gr}.

In general this shows that $M'/p^i$ descends 
to a finite projective graded module $M^i=\bigoplus_nM^i_n$ over $S/p^i$. 
Let $M_n=\varprojlim_iM^i_n$. 
We have to show that the graded $S$-module $M=\bigoplus_n M_n$ 
is finite projective. 
Assume that $p^mR=0$ and let $\bar L=M^{m}\otimes_{S,\rho}R$,
which is independent of $m$. Let $L$ be a finite projective graded
$S_0$-module which lifts $\bar L$ and choose a homomorphism
$N=L\otimes_{S_0}S\to M$ that induces the identity of $\bar L$
under reduction by the homomorphism $\rho:S\to R$ of \eqref{Eq:rho}. 
Then $N/p^i\to M^i$ is an isomorphism by Corollary \ref{Co:NAK-proj},
hence $N\to M$ is an isomorphism, and $M$ is projective.
\end{proof}

\begin{Lemma}
\label{Le:descent-grad-mod-W}
The categories of finite projective graded $\u W(R)$-modules
for varying rings $R$ in which $p$ is nilpotent form an fpqc stack.
\end{Lemma}

\begin{proof}
As in the preceding proof, the homomorphisms form an fpqc sheaf because
$W(R)$ and $I(R)$ form fpqc sheaves.
Assume that $R\to R'$ is faithfully flat, let $R''=R'\otimes_RR'$,
and let $\u S=\u W(R)$, $\u S'=\u W(R')$,  $\u S''=\u W(R'')$.
We have to show that a finite projective graded $S'$-module $M'$
with a descent datum over $S''$ descends to $S$. 

The finite projective graded $R'$-module $\bar L'=M'\otimes_{S',\rho}R'$
carries a descent datum over $R''$ and thus descends to a
finite projective graded $R$-module $\bar L$.
Let $L$ be a finite projective graded $S_0$-module that lifts $\bar L$,
and let $N=L\otimes_{S_0}S$ and $N'=N\otimes_SS'$. 
Using Corollary \ref{Co:NAK-proj} we find
an isomorphism $\alpha:N'\cong M'$ lifting the identity of $\bar L'$.
We have to show that there is an $\alpha$ which commutes with the descent data.
The obstruction lies in the \v Cech cohomology set $H^1(R'/R,K)$
where $K$ is the kernel of $\uAut(N)\to\uAut(\bar L)$,
and $\uAut(N)$ and $\uAut(\bar L)$ are considered as sheaves
on the category of flat affine $R$-schemes.
We will show that $H^1(R'/R,K)$ is trivial.

First, if $R$ is an $\FF_p$-algebra, by an obivous variant of Lemma \ref{Le:AutM}
one finds a representation $K=\varprojlim K_m$
with surjective transition maps whose kernels are vector groups, 
and with $K_0=0$, thus $H^1(R'/R,K)=0$.
In general assume that $p^{i+1}R=0$ and $p^iR\ne 0$ for some $i\ge 1$.
Let $\bar R=R/p^i$, let $\bar N$ and $\bar{\bar L}$ 
be the base change of $N$ and $\bar L$ under $R\to\bar R$,
and let $\bar K$ be the kernel of $\uAut(\bar N)\to\uAut(\bar{\bar L})$,
where $\uAut(\bar N)$ and $\uAut(\bar{\bar L})$ 
are again considered as sheaves
on the category of flat affine $R$-schemes.
The reduction map $K\to\bar K$ is surjective as presheaves.
%
%
Using that the kernel of $W(R)\to W(\bar R)$ has square zero
and is isomorphic to $(p^iR)^\NN$ via $\log$, one verifies that
the kernel $K_0$ of $K\to\bar K$ is isomorphic to an
infinite direct product of quasi-coherent sheaves associated
to $R$-modules of the form $p^iV$ for finite projective $R$-modules $V$.
Thus $H^1(R'/R,K_0)$ is trivial, 
and by induction it follows that $H^1(R'/R,K)$ is trivial as required.
%
%
%
%
\end{proof}

\subsection{Descent of displays}

\begin{Lemma}
Let $\u S$ be a $p$-adic frame for $R$
(Definition \ref{Def:p-adic-frame}).
The categories of displays over $\u\bS(R')$ 
for varying etale $R$-algebras $R'$ form an etale stack.
\end{Lemma}

\begin{proof}
The modules $M$ of a display form an etale stack by
Lemma \ref{Le:descent-grad-mod-S}.
The isomorphisms $M^\sigma\cong M^\tau$ form an etale sheaf by
etale descent for the modules 
$M^\sigma/p^rM^\sigma$ and $M^\tau/p^rM^\tau$
over $S_0/p^r$.
\end{proof}

\begin{Lemma}
The categories of displays over $\u W(R)$ 
for varying rings $R$ in which $p$ is nilpotent form an fpqc stack.
\end{Lemma}

\begin{proof}
The modules $M$ of a display form an fpqc stack by
Lemma \ref{Le:descent-grad-mod-W}.
The isomorphisms $M^\sigma\cong M^\tau$ form an fpqc sheaf by
fpqc descent of finite projective $W(R)$-modules,
see \cite[Cor.~34]{Zink:Display}.
\end{proof}

\section{$G$-displays}

Let $k$ be a finite extension of $\FF_p$.

\begin{Defn}
A display datum is a pair $(G,\mu)$ 
where $G$ is a smooth affine group scheme of finite type over $\ZZ_p$ 
and $\mu$ is a cocharacter of $G_{W(k)}$.

A frame over $W(k)$ is a frame $\u S$ where $S$ is a graded $W(k)$-algebra
and $\sigma:S\to S_0$ extends the Frobenius of $W(k)$.
\end{Defn}

\begin{Example}
If $R$ is a $W(k)$-algebra, then $W(R)$ is a $W(k)$-algebra
by the Cartier homomorphism $W(k)\to W(W(k))$
(see \cite[Chap.~VII, Prop.~4.12]{Lazard})
composed with $W(W(k))\to W(R)$.
It follows that the frames 
$\u W(R)$, $\uW m(R)$, $\u W(B/A)$, $\u\WW(R)$, and $\u\WW(B/A)$
of Examples \ref{Ex:frame-W}, \ref{Ex:frame-Wm}, \ref{Ex:frame-W-rel}, 
\ref{Ex:frame-WW}, and \ref{Ex:frame-WW-rel} 
are frames over $W(k)$ if the input rings $R$ and $B$ are $W(k)$-algebras.
\end{Example}

\begin{Remark}
\label{Rk:G-disp-O}
The following definition of $G$-displays can be extended 
to $\OOO$-frames as in Remark \ref{Rk:O-frames}.
In that context, a display datum 
will be a pair $(G,\mu)$ where $G$ is
an affine group scheme of finite type over $\OOO$ and $\mu$ is a cocharacter
of $G$ defined over a finite unramified extension $\OOO'$ of $\OOO$, 
and one should consider $\OOO$-frames over $\OOO'$ in the obvious sense.
\end{Remark}

\subsection{Definition of the display group}


For an affine $W(k)$-scheme $X=\Spec A$,
an action of $\Gm$ on $X$ is the same as a $\ZZ$-grading of the $W(k)$-algebra $A$.
Explicitly, if an action 
\[
X\times\Gm\to X,\qquad(x,\lambda)\mapsto x*\lambda
\]
is given, then $\lambda\in\Gm$ acts on $f\in A$ by $(\lambda*f)(x)=f(x*\lambda)$,
and $A_n$ is the set of all $f\in A$ such that $\lambda*f=\lambda^nf$.
If $\Gm$ acts on $X$ and $S$ is a $\ZZ$-graded $W(k)$-algebra,
we denote by 
\[
X(S)^{0}\subseteq X(S)
\]
the set of $\Gm$-equivariant morphisms $\Spec S\to X$ over $W(k)$, 
or equivalently the set of homomorphisms of graded $W(k)$-algebras
$A\to S$.


%

Let $(G,\mu)$ be a display datum and let $\u S$ be a frame over $W(k)$ for $R$.
We consider the action of $\lambda\in\Gm$ on $G_{W(k)}$ by 
$g*\lambda=\mu(\lambda)^{-1}g\mu(\lambda)$ 
and define the display group
\begin{equation}
\label{Eq:G(S)-mu}
G(S)_\mu=G_{W(k)}(S)^{0}
\end{equation}
%
with respect to this action.
The homomorphisms $\sigma,\tau:S\to S_0$ of $\ZZ_p$-algebras induce
group homomorphisms 
\[
\sigma,\tau:G(S)_\mu\to G(S_0).
\]
Explicitly, if $G=\Spec A$, an element $g\in G(S)_\mu$
corresponds to a homomorphism of graded $W(k)$-algebras
$g:A\otimes W(k)\to S$,
and $\sigma(g)$ corresponds to the composition
\[
A\xrightarrow{x\mapsto x\otimes 1} 
A\otimes W(k)\xrightarrow {\;g\;} 
S\xrightarrow{\;\sigma\;} S_0,
\]
and similarly for $\tau$.
The group $G(S)_\mu$ acts on the set $G(S_0)$ by
\begin{equation}
\label{Eq:action}
G(S_0)\times G(S)_\mu\to G(S_0),
\qquad
(x,g)\mapsto \tau(g^{-1})x\sigma(g).
\end{equation}
We want to take the quotient of \eqref{Eq:action}
with respect to the etale topology of $R$.
This is possible when $\u S$ is a $p$-adic frame or when $\u S=\u W(R)$
since in both cases $\u S$ is a functor of $R$, at least for etale maps.
Before carrying this out let us record how the action \eqref{Eq:action}
is related to $\sigma$-conjugation.

\subsection{$G$-displays and $\sigma$-$\mu$-conjugation}
\label{Se:G-disp-tf}


Let $(G,\mu)$ be a display datum
and $\u S$ a frame over $W(k)$.
We define a pre-frame $\u B=\u S{}_{\iso}$ 
with a homomorphism $\u S\to\u B$ by $B_0=S_0[1/p]$ and $B=B_0[t,t^{-1}]$,
which determines $\sigma,\tau:B\to B_0$.
Let $G(B)_\mu\subseteq G(B)$ as in \eqref{Eq:G(S)-mu}.
Again we obtain group homomorphisms $\sigma,\tau:G(B)_\mu\to G(B_0)$.

\begin{Lemma}
\label{Le:action-tf}
The homomorphism $\tau:G(B)_\mu\to G(B_0)$ is bijective,
and under this isomorphism the action 
$G(B_0)\times G(B)_\mu\to G(B_0)$ defined as in \eqref{Eq:action}
corresponds to
\begin{equation}
\label{Eq:action2}
G(B_0)\times G(B_0)\to G(B_0),\qquad 
(x,g)\mapsto g^{-1}\cdot x\cdot\sigma_0(\mu(p)g\mu(p)^{-1}).
\end{equation}
\end{Lemma}

\begin{proof}
Let $\tau_p:B\to B_0$ be the homomorphism of $B_0$-algebras 
with $\tau_p(t)=p$, and recall that $\tau(t)=1$. 
Then $\sigma=\sigma_0\circ\tau$.
We can identify $\Spec B=\Gm\times\Spec B_0$ 
such that $\lambda\in\Gm$ acts on $\Gm$ 
by multiplication with $\lambda^{-1}$
since $\deg(t)=-1$. 
Hence $\tau$ and $\tau_p$ induce bijective homomorphisms
$G(G)_\mu\to G(B_0)$, 
and the automorphism $\tau_p\circ\tau^{-1}$ of $G(B_0)$
is given by the action of $p^{-1}\in\Gm(B_0)$,
which ist $x\mapsto\mu(p)x\mu(p)^{-1}$.
Hence the endomorphism $\sigma\circ\tau^{-1}$ of $G(B_0)$
is given by $x\mapsto\sigma_0(\mu(p)x\mu(p)^{-1})$,
and the lemma follows.
\end{proof}

If the ring $S_0$ is torsion free, 
$S_0\to B_0$ is injective,
and $G(S_0)$ is a subgroup of $G(B_0)$.
In this case the action \eqref{Eq:action2}
determines the action \eqref{Eq:action}.

\subsection{$G$-displays over $p$-adic frames}
\label{Se:G-disp-nil-frame}

Let $\u S$ be a $p$-adic frame over $W(k)$ for $R$
and let $\u\bS$ be the associated sheaf of frames
on the category of affine etale schemes over $\Spec R$ 
as in \S \ref{Se:etale-loc}.

\begin{Lemma}
\label{Le:X(S)-sheaf}
If $X=\Spec A$ is an affine $W(k)$-scheme of finite type
with an action of\/ $\Gm$,
the presheaves $X(\bS)^{0}$ and $X(\bS_0)$ on the
category of affine etale $R$-schemes are etale sheaves.
%
\end{Lemma}

\begin{proof}
If $R\to R'$ is an etale homomorphism 
and if $\u S\to\u S'$ is the base change of $\u S$, 
then $X(\bS)^{0}(R')=X(S')^{0}$ and $X(\bS_0)(R')=X(S'_0)$.
Assume that $X=\AA^1=\Spec W(k)[u]$ with weight $r\in\ZZ$,
i.e.\ $u$ has degree $r$, 
which corresponds to the action of $\Gm$ by $\lambda*u=\lambda^ru$.
Then $X(S)^{0}=S_r$ and $X(S_0)=S_0$,
and the sheaf property holds by Lemma \ref{Le:S-sheaf}.
In general, $X$ is the equalizer of a pair of maps $\AA^n\rightrightarrows\AA^m$,
where each coordinate is homogeneous of some degree.
Hence $X(\bS)^{0}$ 
and $X(\bS_0)$ 
are equalizers of maps between etale sheaves,
so they are etale sheaves.
\end{proof}

\begin{Defn}
\label{Def:G-disp-nil}
For a display datum $(G,\mu)$,
the groupoid of $G$-displays of type $\mu$ over $\u S$ is defined as
\[
G\GDisp_\mu(\u S)=[G(\bS_0)/G(\bS)_\mu](\Spec R),
\]
the quotient groupoid of the action \eqref{Eq:action} 
with respect to the etale topology.
\end{Defn}

Thus $G\GDisp_\mu(\u S)$
is the groupoid of pairs $(Q,\alpha)$ where $Q$ is an
etale $G(\bS)_\mu$-torsor over $\Spec R$ and 
$\alpha:Q\to G(\bS_0)$ is a $G(\bS)_\mu$-equivariant map of sheaves
with respect to the action \eqref{Eq:action}.

\begin{Remark}
\label{Rk:G-disp}
$G\GDisp_\mu(\u S)$
is also equivalent to the groupoid of pairs $(Q,\gamma)$ 
where $Q$ is an etale $G(\bS)_\mu$-torsor over $\Spec R$ and 
$\gamma:Q^\sigma\to Q^\tau$ is an isomorphism of $G(\bS_0)$-torsors.
Indeed, the map of sheaves 
$G(\bS_0)\times G(\bS_0)\to G(\bS_0)$,
$(g,g')\mapsto g^{-1}g'$ 
induces an isomorphism
\[
[G(\bS_0)\backslash(G(\bS_0)\times G(\bS_0))/G(\bS)_\mu]\xrightarrow\sim
[G(\bS_0)/G(\bS)_\mu]
\]
where the left quotient groupoid is taken with respect to the action
\begin{multline*}
G(\bS_0)\times(G(\bS_0)\times G(\bS_0))\times G(\bS)_\mu\to
G(\bS_0)\times G(\bS_0),\\
(h,(g,g'),k)\mapsto (hg\tau(k),hg'\sigma(k)),
\end{multline*}
which leads to the pairs $(Q,\gamma)$.
%
\end{Remark}

\begin{Remark}[Functoriality]
For a homomorphism $\u S\to\u S'$ of $p$-adic frames  
there is a base change functor $G\GDisp_\mu(\u S)\to G\GDisp_\mu(\u S')$.
For a homomorphism $\gamma:G\to G'$ of group schemes of finite type 
over $\ZZ_p$ and $\mu'=\gamma\circ\mu$
there is a base change functor $G\GDisp_\mu(\u S)\to G'\GDisp_{\mu'}(\u S)$.
\end{Remark}

\begin{Example}
\label{Ex:disp-GLn}
Let $G=\GL_n$ and let $\mu:\Gm\to\GL_n$ be the cocharacter 
$\mu(x)=\diag(x^{\mu_1},\ldots,x^{\mu_n})$ with $\mu_1\ge\ldots\ge\mu_n$.
Then $x\in\Gm$ acts on a matrix $A=(a_{ij})$ 
by $a_{ij}\mapsto x^{\mu_j-\mu_i}a_{ij}$. 
It follows that $\GL_n(S)_\mu=\Aut_S^{0}(N_\mu)$ 
for the graded $S$-module 
$N_\mu=\bigoplus _iS(-\mu_i)$; see Lemma \ref{Le:AutM}.
Every graded $S$-module of type $\mu$ is isomorphic to $N_\mu$ locally
in $\Spec R$. Hence \S \ref{Se:disp-change-basis}
implies that
\[
\GL_n\GDisp_{\mu}(\u S)
\]
is the groupoid of displays of type $\mu$ over $\u S$.
\end{Example}

\begin{Example}
Let $R$ be an admissible local Artin
$W(k)$-algebra as in Example \ref{Ex:frame-WW},
or let $B\to A$ be a PD thickening of admissible 
local Artin $W(k)$-algebras
with nilpotent divided powers.
We denote the groupoids of $G$-displays of type $\mu$ over
the associated frames $\u\WW(R)$ and $\u\WW(B/A)$ by
\[
G\GDisp_\mu^\WW(R)=G\GDisp_\mu(\u\WW(R))
\]
and
\[
G\GDisp_\mu^\WW(B/A)=G\GDisp_\mu(\u\WW(B/A)).
\]
\end{Example}

\begin{Example}
For a $k$-algebra $R$ and $m\ge 1$ we denote the 
groupoid of $G$-displays of type $\mu$ 
over the frame $\uW m(R)$ of Example \ref{Ex:frame-Wm} by
\[
G\GDisp_\mu^{W_m}(R)=G\GDisp_\mu(\uW m(R)).
\]
The functoriality of $\uW m(R)$ with respect to $R$
and the resulting base change of $G$-displays 
makes $G\GDisp_\mu^{W_m}$ into a fibered category over
the category of affine $\FF_p$-schemes,
and $G\GDisp_\mu^{W_m}$ is a smooth algebraic stack 
of dimension zero over $k$ because the functors 
$
R\mapsto G(\uW m(R))_\mu
$
and
$
R\mapsto G(W_m(R))
$
are representable by affine smooth group schemes of equal dimension;
see Lemma \ref{Le:Greenberg-graded} below, which extends
\cite{Greenberg:Schemata, Greenberg:SchemataII} to the graded setting.
\end{Example}

\begin{Lemma}
\label{Le:Greenberg-graded}
Let $X=\Spec A$ be an affine $W(k)$-scheme of finite type
with a $\Gm$-action, and $n\ge 1$.
The functor of $k$-algebras $X_n(R)= X(\uW n(R))^{0}$ 
is representable by an affine $k$-scheme of finite type.
If $X$ is smooth of dimension $d$, then $X_n$ is smooth of dimension $nd$,
and the reduction map $X_{n+1}\to X_n$ 
is surjective and smooth of dimension $d$.
\end{Lemma}

\begin{proof}
Assume that $X=\AA^1$ with weight $r\in\ZZ$.
Then $X_n(R)=W_n(R)$ when $r\le 0$ and $X_n(R)=I_{n+1}(R)$ when $r>0$,
in both cases $X_n\cong\AA^n$ such that $X_{n+1}\to X_n$ is a
coordinate projection $\AA^{n+1}\to\AA^n$. 
In general, $X$ is the equalizer of a pair of maps $\AA^m\rightrightarrows\AA^l$
where each coordinate is homogeneous of some degree,
hence $X_n$ is representable affine of finite type.
Assume that $X$ is smooth of dimension $d$.
Then $X$ is covered by open sets of the form $U=\Spec A_f$ for a
homogeneous $f\in A$ such that there is a $\Gm$-equivariant etale
map $U\to\AA^d=Y$ 
where each coordinate of $\AA^d$ is homogeneous of some degree.
The resulting functors $U_n$ form an open cover of $X_n$.
We claim that $U_n\to Y_n\cong\AA^{dn}$ is etale, which proves the lemma.
To prove the claim let $S\to R$ be a surjective homomorphism of $\FF_p$-algebras
with kernel of square zero and let $u\in U_n(R)$ and $y\in Y_n(S)$ be
given with equal image in $Y_n(R)$. 
The kernel of the ring homomorphism $\uW n(S)\to \uW n(R)$ has square
zero, which implies that there is a unique $\tilde u\in U(\uW n(S))$ 
which lifts $u$ and $y$. 
The element $\tilde u$ is $\Gm$-equivariant by its uniqueness,
which means that $\tilde u\in U_n(S)$, and the claim is verified.
\end{proof}

\begin{Remark}
If $(G,\mu)$ is of Hodge type, i.e.\ there is an embedding $G\to\GL_n$ such that $\mu$ is minuscule for $\GL_n$, a related construction of truncated displays with $(G,\mu)$-structure is given in \cite{Zhang:Stratifications}.
\end{Remark}

\subsection{$G$-displays over Witt vectors}
\label{Se:G-disp-W}

Let $(G,\mu)$ be a display datum.
We have fpqc sheaves $G(W)$ and $G(\u W)_\mu$
on the category $\AffNilp_{W(k)}$ 
of affine $W(k)$-schemes on which $p$ is nilpotent
defined by 
\[
G(W)(\Spec A)=G(W(A))
\qquad\text{and}\qquad
G(\u W)_\mu(\Spec A)=G(\u W(A))_\mu
\]
as in \eqref{Eq:G(S)-mu};
the sheaf property follows from Lemma \ref{Le:XWRgr} below.
We define a fibered category over $\AffNilp_{W(k)}$,
\begin{equation}
G\GDisp_\mu^W=[G(W)/G(\u W)_\mu],
\end{equation}
the quotient groupoid of the action defined by \eqref{Eq:action} 
with respect to the etale topology.
By Lemma \ref{Le:G-disp-W-fpqc} below
one could also take the fpqc topology. 
The groupoid $G\GDisp_\mu^W(R)=G\GDisp_\mu^W(\Spec R)$ will be called the
groupoid of $G$-displays over $\u W(R)$ of type $\mu$.

\begin{Lemma}
\label{Le:XWRgr}
Let $X=\Spec A$ be an affine $W(k)$-scheme 
with an action of\/ $\Gm$.
The functors of $W(k)$-algebras $R\mapsto X(\u W(R))^{0}$ 
and $R\mapsto X(W(R))$ are representable.
\end{Lemma}

\begin{proof}
This follows from the proof of Lemma \ref{Le:Greenberg-graded},
using that the functors $R\mapsto W(R)$ and $R\mapsto I(R)$ are representable
and that the equalizer of a pair of maps between representable
sheaves is representable. 
\end{proof}

\begin{Lemma}
\label{Le:G-disp-W-fpqc}
The fibered category $G\GDisp_\mu^W$ is an fpqc stack.
\end{Lemma}

\begin{proof}
Let $K=G(\u W)_\mu$, 
which is an affine group scheme over $W(k)$ by Lemma \ref{Le:XWRgr}.
We show that every fpqc locally trivial $K$-torsor over a 
$W_r(k)$-algebra $R$ is etale locally trivial.
Let $S(R)$ be the graded ring of the frame $\u W(R)$.
For $m\ge 1$ let $J_{(m)}(R)\subseteq S(R)$ 
be the graded ideal given by $v^{m}W(R)$ in every degree.
One verifies easily that this is an ideal, 
using the relation
$v(a)* v^{m}(b)=v(av^{m-1}(b))=v^m(\sigma^{m-1}(a)b)$
in $W(R)$, where $*$ is the product in $S(R)$ in positive degrees.
The quotient $S_{(m)}(R)=S(R)/J_{(m)}(R)$ is a graded ring, 
and $K(R)=G(S(R))_\mu$ is the limit over $m$ of $K_m(R)=G(S_{(m)}(R))_\mu$.
The proof of Lemma \ref{Le:Greenberg-graded} shows that each $K_m$
is a smooth affine group scheme over $W(k)$ of finite type,
and the reduction homomorphism $\pi_m:K_{m+1}\to K_m$ is smooth and surjective.
Over $W_{m-1}(k)$ the kernel of $\pi_m$ can be identified with
the vector group $\Lie(G)$ because if $p^{m-1}R=0$ then
the ideal $J_{(m)}(R)/J_{(m+1)}(R)$ has square zero, 
using the relation $v(v^{m-1}(a)v^{m-1}(b))=p^{m-1}v^{m}(ab)$.
Now if $Q$ is an fpqc $K$-torsor over a $W_r(k)$-algebra $R$,
the associated $K_{r+1}$-torsor $Q_{r+1}$ is smooth over $R$
and thus etale locally trivial. 
The projection $Q_{m+1}\to Q_m$ is a torsor under $\Ker(\pi_m)$,
which is trivial when $m\ge r+1$ since $\Ker(\pi_m)$ is a vector
group and $Q_m$ is affine. Hence if $Q_{r+1}$ is trivial then so is $Q$.
\end{proof}

\begin{Remark}
If $R$ is a $k$-algebra, then 
$G\GDisp^W_\mu\cong\varprojlim_mG\GDisp_\mu^{W_m}$.
Indeed, the groups $G(\u W(R))_\mu$ and $G(W(R))$ are the limits 
of $G(\uW m(R))_\mu$ and $G(W_m(R))$, and one verifies
easily that a $G(\u W)_\mu$-torsor (etale or equivalently fpqc)
is the same as a compatible system of $G(\uW m)_\mu$-torsors.
\end{Remark}

There is also a relative version.
Let $B\to A$ be a PD thickening of $W(k)$-algebras in which
$p$ is nilpotent. 
For an etale $A$-algebra $A'$ let $B'$ be the unique etale $B$-algebra
with $B'\otimes_BA=A'$, and let us write $\bB(A')=B'$ and $\bA(A')=A'$.
Then $\u W(\bB/\bA)$ is an etale sheaf of frames on the category of
affine etale $A$-schemes.
We define a fibered category over this category
\begin{equation}
G\GDisp_\mu^{W(B/A)}=[G(W(\bB/\bA))/G(\u W(\bB/\bA))_\mu],
\end{equation}
the quotient groupoid of the action \eqref{Eq:action} 
with respect to the etale topology,
and call $G\GDisp_\mu(\u W(B/A))=G\GDisp_\mu^{W(B/A)}(\Spec A)$
the groupoid of $G$-displays of type $\mu$ over $\u W(B/A)$.

\subsection{Orthogonal displays}
\label{Se:orth-disp}

Let $\u S$ be a frame for $R$ and assume that $p\ge 3$. 

\begin{Defn}
\label{Def:orth-disp}
An orthogonal graded module over $S$ is a finite projective graded $S$-module
$M$ with a perfect symmetric bilinear form $\beta:M\times M\to S$
of degree zero.
An orthogonal display over $\u S$ is a display $\u M$ with
a perfect symmetric bilinear form $\beta:\u M\times\u M\to \u S$. 
\end{Defn}

We fix $n$, the rank of $M$.
Let $\mu=(\mu_1,\ldots,\mu_n)\in\ZZ^n$ with $\mu_1\ge\ldots\ge\mu_n$.
We call $\mu$ an orthogonal type if $\mu_i+\mu_{n+1-i}=0$ for all $i$.
The type of an orthogonal graded module or display is orthogonal.

Let $\psi:\ZZ^n\times\ZZ^n\to\ZZ$ be the symmetric bilinear form 
$\psi(x,y)=x^t Jy$ with 
\[
J=\left(\begin{smallmatrix}&& 1\\ & \iddots \\ 1\end{smallmatrix}\right).
\]
Let $G=\Ogr(\psi)\subseteq\GL_n$ be the orthogonal group of $\psi$,
so $G(R)$ is the set of all $A\in M_n(R)$ with $A^tJA=J$. 
For an orthogonal type $\mu$, the associated cocharacter 
\[
\mu:\Gm\to\GL_n,\qquad
x\mapsto\diag(x^{\mu_1},\ldots,x^{\mu_n})
\]
factors as $\mu:\Gm\to G=\Ogr(\psi)$. 

\begin{Prop}
\label{Pr:orth-disp}
If $\u S$ is a $p$-adic frame as in Definition \ref{Def:p-adic-frame},
the category of orthogonal displays of type $\mu$ over $\u S$
is equivalent to the category of $\Ogr(\psi)$-displays 
of type $\mu$ over $\u S$ as in Definition \ref{Def:G-disp-nil}.
\end{Prop} 

The proof is quite formal, using the following two lemmas.

\begin{Lemma}
\label{Le:N-psi}
Let $N_\mu=\bigoplus_iS(-\mu_i)$ as a graded $S$-module,
equipped with the bilinear form $\psi:N_\mu\times N_\mu\to S$ 
defined by $\psi(a,b)=a^t Jb$.
Then 
\[
G(S)_\mu=\Aut^{0}_S(N_\mu,\psi),
\]
the group of automorphisms of the graded module that preserve $\psi$;
see \eqref{Eq:G(S)-mu} for the definition of $G(S)_\mu$.
Moreover, $G(S_0)$ is bijective to the set 
$\FFF$ of $\sigma$-linear maps $F:N_\mu\to (N_\mu)^\tau$ 
such that $(N_\mu,F,\psi)$ is an orthogonal display. 
The conjugation action of $g\in\Aut^{0}_S(N_\mu,\psi)$ 
on $\FFF$
by $F\mapsto \tau(g)^{-1}\circ F\circ g$ 
corresponds to the action \eqref{Eq:action}.
\end{Lemma}

\begin{proof}
We have $G(S)=\Aut_S(N_\mu,\psi)$ and
$\GL_n(S)_\mu=\Aut_S^{0}(N_\mu)$ 
as in Example \ref{Ex:disp-GLn},
hence
\[
G(S)_\mu=G(S)\cap\GL_n(S)_\mu
=\Aut_S(N,\psi)\cap\Aut_S^{0}(N_\mu)
=\Aut_S^{0}(N_\mu,\psi),
\]
which proves the first assertion.
The set $\FFF$ is bijective to the set of isomorphisms 
of orthogonal $S_0$-modules $(N_\mu^\sigma,\psi)\cong(N_\mu^\tau,\psi)$.
But $N_\mu=S^n$ as an $S$-module and thus $N_\mu^\tau=S_0^n=N_\mu^\sigma$,
and the bilinear form $\psi$ is given by $J$ on these modules.
The lemma follows.
\end{proof}

\begin{Lemma}
\label{Le:orth-loc-isom}
If $\u S$ is a $p$-adic frame,
two orthogonal $S$-modules of the same type are isomorphic 
etale locally in $\Spec R$, i.e.\ they become isomorphic over
$S'=\bS(R')$ for a faithfully flat etale homomorphism $R\to R'$,
using the notation of\/ \S \ref{Se:etale-loc}. 
\end{Lemma}

\begin{proof}
It suffices to show that any orthogonal $S$-module $(M,\beta)$ of type $\mu$
is isomorphic to $(N_\mu,\psi)$ etale locally in $\Spec R$.
After localizing $R$ we may assume that $M=\bigoplus_iS(-\mu_i)$ 
as a graded $S$-module, thus $M^*=\bigoplus_iS(\mu_i)$.
Assume first that $d:=\mu_1>0$. 
Let $r$ be the multiplicity of $d$ in $\mu$ and write
$M=M_1\oplus M_2\oplus M_3$ as graded modules
with $M_1\cong S(-d)^r$ and $M_2\cong S(d)^r$.
The perfect bilinear form $\beta$ corresponds to an isomorphism
$\beta':M\xrightarrow\sim M^*$. The composition
\begin{equation}
\label{Eq:incl-beta-pr}
M_1\xrightarrow{incl} M\xrightarrow{\beta'} M^*\xrightarrow{pr}(M_2)^*
\end{equation}
becomes an isomorphism under the base change by $\rho:S\to R$,
thus \eqref{Eq:incl-beta-pr} is an isomorphism by Corollary \ref{Co:NAK-proj}.
We can choose homogeneous elements $e_1\in (M_1)_d$ 
and $e_2\in (M_2)_{-d}$ with $\beta(e_1,e_2)=1$.
We claim that they can be modified inside $M_d$ and $M_{-d}$ 
such that in addition $\beta(e_1,e_1)=\beta(e_2,e_2)=0$.

First, there is a unique $x\in S_{-2d}$ such that $e_2'=e_2+xe_1$ 
satisfies $\beta(e_2',e_2')=0$. 
Indeed, let $a=\beta(e_1,e_1)\in S_{2d}$ and $c=\beta(e_2,e_2)\in S_{-2d}$.
We have to solve the equation $ax^2+2x+c=0$ in $S_{-2d}$.
If we write $x=t^{2d}x_0$ and $c=t^{2d}c_0$ and $a_0=t^{2d}a$,
this is equivalent to $a_0x_0^2+2x_0+c_0=0$ in $S_0$.
Since $a_0\in tS_1$ and since the image of $tS_1$ in $S_0/p^m$ 
is nilpotent for each $m$,
there is a unique solution $x_0$; recall that $p\ge 3$.
Now $\beta(e_1,e_2')=ax+1$ is a unit since $ax\in tS_1$.
Thus after modifying $e_1$ and $e_2$ we can assume that 
$\beta(e_1,e_2)=1$ and $\beta(e_2,e_2)=0$.
Finally, there is a unique $y\in S_{2d}$ such that $e_1'=e_1+ye_2$ 
satisfies $\beta(e_1',e_1')=0$. Indeed, if again $a=\beta(e_1,e_2)$,
this is equivalent to $a+2y=0$ in $S_{2d}$. This proves the claim.

Let $M'=Se_1\oplus Se_2$. It follows that $(M',\beta)$ is isomorphic to
$(S(-d)\oplus S(d),\psi)$ where $\psi$ is the standard form given by
the $2\times 2$-version of $J$.
Since the composition 
\[
M'\xrightarrow i M\xrightarrow{\beta'}M^*\xrightarrow{i^*}M'^*
\]
is an isomorphism, we have an orthogonal decomposition $M=M'\oplus M'^\perp$.

By induction the lemma is reduced to the case $d=0$, i.e.\ $\mu=(0,\ldots,0)$.
Thus $M=S^n$, and $\beta$ is given by a symmetric invertible matrix $B\in M_n(S_0)$.
After etale localization we may assume that $B$ maps to $J$ under 
$M_n(S_0)\to M_n(R)$.
Let $B=J+C$, let $A=1-JC/2$, and form $B'=A^tBA=J+C'$.
If $\Fa\subset S_0$ is the ideal generated by the coefficients of $C$,
then the coefficients of $C'$ lie in $\Fa^2$. 
Since $\Fa$ maps to a nilpotent ideal in each $S_0/p^m$,
it follows that the iteration of $B\mapsto B'$ converges to $J$, and the lemma is proved.
\end{proof}

\begin{proof}[Proof of Proposition \ref{Pr:orth-disp}]
Let $(M,\beta)$ be an orthogonal graded module of type $\mu$ over $\u S$,
let $(\bM,\beta)$ be the etale sheaf of orthogonal graded $\bS$-modules
defined by $\bM(R')=M\otimes_S\bS(R')$,
and let 
$
Q=\uIsom((\bN_\mu,\psi),(\bM,\beta))
$
be the etale sheaf of isomorphisms of orthogonal graded modules.
Then $Q$ is a $G(\bS)_\mu$-torsor by Lemma \ref{Le:orth-loc-isom},
and the assignment $(M,\beta)\mapsto Q$ is an equivalence by
etale descent of finite projective graded modules, 
Lemma \ref{Le:descent-grad-mod-S}.
An extension of $(M,\psi)$ to an orthogonal display $(M,F,\psi)$
corresponds to an isomorphism of orthogonal modules
$(M^\tau,\psi)\cong(M^\sigma,\psi)$ over $S_0$, 
which is the same as an isomorphism of the associated
$G(\bS_0)$-torsors $\gamma:Q^\tau\cong Q^\sigma$,
and $(Q,\gamma)$ is a $G$-display of type $\mu$;
see Remark \ref{Rk:G-disp}.
\end{proof}

\begin{Remark}
Similarly one verifies that 
orthogonal displays over $\u W(R)$ are equivalent
to $\Ogr(\psi)$-displays over $\u W(R)$ as in \S \ref{Se:G-disp-W}.
\end{Remark}

\section{Structure of the display group}

Let $(G,\mu)$ be a display datum 
with $G=\Spec A'$, and $A=A'\otimes_{\ZZ_p}W(k)$.

\subsection{Weight subgroups}

Recall that $x\in\Gm$ acts on $G_{W(k)}$ by $g\mapsto\mu(x)^{-1}g\mu(x)$. 
We consider the maximal subgroup schemes 
$P^+$ and $P^-$ of $G_{W(k)}$ on which $\Gm$
acts by non-negative resp.\ non-positive weights. 
The intersection $L=P^+\cap P^-$ is the centralizer of $\mu$.
Explicitly, 
let $A=\bigoplus_{n\in\ZZ} A_n$ be the weight decomposition for the $\Gm$-action,
let $(A_{>0})$ and $(A_{<0})$ be the ideals generated by the elements of
positive or negative degree, and define
\begin{equation}
\label{Eq:Ppm}
A^+=A/(A_{<0}),\qquad A^-=A/(A_{>0}),\qquad \bar A=A/(A_{>0}\cup A_{<0}).
\end{equation}
Then $P^{\pm}=\Spec A^{\pm}$ and $L=\Spec\bar A$, 
and these are indeed subgroups of $G$ because the co-multiplication
$A\to A\otimes A$ preserves the degree and thus maps $A_{>0}$ into
$A_{>0}\otimes A+A\otimes A_{>0}$, and similarly for $A_{<0}$.

There are unique $\Gm$-equivariant homomorphisms $P^{\pm}\to L$
which extend the identity of $L$.
Namely, since the $\ZZ$-graded rings $A^+$ and $A^-$ 
are concentrated in non-negative resp.\ non-positive degrees,
we have $(A^+)_0=A^+/(A^+)_{>0}=\bar A$.
Let $U^{\pm}\subseteq P^{\pm}$ be the kernel of $P^\pm\to L$.
The multiplication map $U^{\pm}\times L\to P^{\pm}$ is an isomorphism of schemes.

The group schemes $P^\pm$ and $U^\pm$ and $L$ are the group schemes
$P_G(\mp\mu)$ and $U_G(\mp\mu)$ and $Z_{G_{W(k)}}(\mu)$ of \cite{CGP}.
Since $G$ is smooth, the multiplication map 
\[
U^-\times L\times U^+\to G_{W(k)}
\]
is an open immersion by \cite[Proposition 2.1.8 (3)]{CGP}, 
in particular the group schemes $L$ and $U^\pm$ and $P^\pm$
over $W(k)$ are smooth.

\begin{Lemma}
\label{Le:log-U}
Let $(G,\mu)$ be a display datum.
There are $\Gm$-equi\-variant morphisms 
$\log^\pm: U^\pm\to V(\Lie U^\pm)$ 
which induce the identity on the Lie algebras.
These morphisms are necessarily isomorphisms of schemes.
\end{Lemma}

Note that every $\Gm$-equivariant morphism $U^\pm\to V(\Lie U^\pm)$ preserves the
unit section and therefore induces a map of the Lie algebras.
In general the morphisms $\log^\pm$ need not be unique; 
cf.\ Lemma \ref{Le:log-minuscule} below.

\begin{proof}
By symmetry it suffices to consider $U^+$.
We have $U^+=\Spec B$ for a graded $W(k)$-algebra $B$ with $B_0=W(k)$
and augmentation ideal $I=B_{>0}$.
The $W(k)$-module $I/I^2$ is free because $U^+$ is smooth.
Let $B'=\Sym^*(I/I^2)$.
The grading of $I$ induces a grading of $I/I^2$,
and $B'$ becomes a graded $W(k)$-algebra with
with $B'_0=W(k)$. 
Let $I'=B'_{>0}$. Then $I'/I'^2=I/I^2$.

A $\Gm$-equivariant morphism $\log:U^+\to V(\Lie U^+)$ 
corresponds to a
homomorphism of graded $\ZZ_p$-algebras $\lambda:B'\to B$,
which corresponds to a
linear map of graded $W(k)$-modules $s:I/I^2\to I$,
and $\lambda$ induces the identity on the Lie algebras 
iff $s$ is a section of the projection $I\to I/I^2$.
In this case $\lambda$ is surjective:
If $n\ge 1$ is given such that $B'_m\to B_m$ is surjective for all $m<n$,
then the degree $n$ component of $I^2$ lies in the image of $\lambda$,
and it follows that $B'_n\to B_n$ is surjective. 
Thus $\log=\Spec\lambda$ is a closed immersion of a smooth scheme
into an affine space over $\ZZ_p$ 
which induces a bijective map on the Lie algebras.
It follows that $\log$ is an isomorphism.
\end{proof}

\subsection{Decomposition of the display group}

Let $\u S$ be a frame over $W(k)$.
For a subgroup scheme $H\subseteq G_{W(k)}$ 
which is stable under the
conjugation action of $\Gm$ via $\mu$
we consider the group of $\Gm$-equivariant sections $H(S)_\mu=H(S)^{0}$.
Then
\[
P^\pm(S)_\mu=U^\pm(S)_\mu\times L(S)_\mu=U^\pm(S)_\mu\times L(S_0).
\]
The ring homomorphism $\tau:S\to S_0$ induces $\tau:H(S)_\mu\to H(S_0)$.

\begin{Lemma}
\label{Le:tau-GSmu-GS0}
If $H\subseteq P^-$ then
$\tau:H(S)_\mu\to H(S_0)$ is bijective.
The inverse image of\/ $U^+(S_0)$ under $\tau:G(S)_\mu\to G(S_0)$ 
is equal to $U^+(S)_\mu$.
\end{Lemma}

\begin{proof}
Let $\tilde S=S_0[t,t^{-1}]$ as a graded $S_0$-algebra with $t$
in degree $-1$. The ring homomorphism $\tau:S\to S_0$ factors into
homomorphisms of $S_0$-algebras
\[
S\xrightarrow{\;\tilde\tau\;}\tilde S\xrightarrow{\;\alpha\;} S_0
\]
where $\tilde\tau$ preserves the grading and $\tilde\tau(t)=t$, while $\alpha(t)=1$. 
For every subgroup scheme $H\subseteq G_{W(k)}$ 
stable under the action of $\Gm$
the homomorphism $\alpha:H(\tilde S)_\mu\to H(S_0)$ is bijective.
Since $\tilde\tau:S\to\tilde S$ is bijective in non-negative degrees, 
for $H\subseteq P^-$ the homomorphism 
$\tilde\tau:H(S)_\mu\to H(\tilde S)_\mu$ is bijective,
which proves the first assertion.
The augmentation ideal $I\subseteq A$ is $\ZZ$-graded, 
and the subscheme $U^+\subseteq G$ is defined by the ideal generated
by $I_{\le 0}$.
Since $\tilde\tau$ is bijective in non-negative degrees it follows that
an element $g\in G(S)_\mu$ lies in $U^+(S)_\mu$ iff $\tilde\tau(g)$ lies
in $U^+(\tilde S)_\mu$.
\end{proof}

\begin{Prop}
\label{Pr:decomp}
Let $(G,\mu)$ be a display datum 
and let $\u S$ be a frame over $W(k)$. 
Then the multiplication map 
$P^-(S)_\mu\times U^+(S)_\mu\to G(S)_\mu$ is bijective.
\end{Prop}

\begin{proof}
Since $P^-\times U^+\to G_{W(k)}$ is an open immersion, 
the map of the lemma is injective. To prove surjectivity
let $\pi:S_0\to R=S_0/tS_1$ be the projection and consider the
group homomorphisms
\[
G(S)\xrightarrow\tau G(S_0)\xrightarrow{\pi}G(R).
\]
Here $\pi\circ\tau$ maps $G(S)_\mu$ into $P^-(R)$.
Indeed, for $g\in G(S)_\mu$, corresponding to a homomorphism of graded rings
$g^\sharp:A\to S$, the composition $\pi\circ\tau\circ g^\sharp$ 
maps $A_{>0}$ to zero because $\pi\circ\tau$ annihilates $S_{>0}$,
and $P^-$ is defined by the ideal $(A_{>0})$ of $A$.
Since $P^-\times U^+$ is an open subscheme of $G_{W(k)}$ 
and $\pi$ induces a bijective map $\Max(R)\to\Max(S_0)$,
it follows that $\tau$ maps
$G(S)_\mu$ into $P^-(S_0)\times U^+(S_0)$.
Now, using the first part of Lemma \ref{Le:tau-GSmu-GS0}
for $H=P^-$, 
every $g\in G(S)_\mu$ can be written as $g=qh$ with $q\in P^-(S)_\mu$ 
and $h\in G(S)_\mu$ such that $\tau(h)\in U^+(S_0)$,
which implies that $h\in U^+(S)_\mu$ 
by the second part of Lemma \ref{Le:tau-GSmu-GS0}.
\end{proof}

\begin{Example}
[$G$-displays over $\uW 1(R)$ and $G$-zips] 
\label{Ex:G-zip}
Let $R$ be a $k$-algebra and $\u S=\uW 1(R)$; 
see Examples \ref{Ex:frame-W1} \& \ref{Ex:disp-W1-zip}.
In this case the decomposition $S=R[t]\times_{\Frob,R}R[u]$
of Example \ref{Ex:frame-W1} induces a decomposition
\[
G(S)_\mu\cong G(R[t])_\mu\times_{\Frob,G(R_\sigma)_\mu}G(R_\sigma[u])_\mu
=P^-(R)\times_{\Frob,L^\sigma(R)}(P^+)^\sigma(R),
\]
where $R_\sigma$ is the ring $R$ with $k$-algebra strucure via 
the Frobenius $\sigma:k\to k$,
moreover $R[t]$ and $R_\sigma[u]$ and $R_\sigma$ are considered as graded rings
with $R$ in degree zero and $-\deg(t)=1=\deg(u)$.
This decomposition can also be viewed 
as a consequence of Proposition \ref{Pr:decomp}.
It follows that a $G(S)_\mu$-torsor is the same as a triple
$(X^-,X^+,\alpha)$ where $X^-$ is a $P^-$-torsor,
$X^+$ is a $(P^+)^\sigma$-torsor, 
and $\alpha:(X^-/U^-)^\sigma\cong X^+/U^+$ is $L$-equivariant.
It follows that $G$-displays over $\uW 1(R)$ of type $\mu$
are $G$-zips over $R$ of type $\mu^{-1}$ in the sense of 
\cite{Pink-Wedhorn-Ziegler:F-zips}.
In fact, here we reproduce only the case $q=p$ of loc.cit.
For the general case one needs the $\OOO$-frame $\uW {\OOO,1}$
where $\OOO$ is an extension of $\ZZ_p$ with residue field $\FF_q$;
see Remark \ref{Rk:frames-WO}.
\end{Example}

\subsection{$1$-bounded and minuscule display data}


\begin{Defn}
\label{Def:minuscule}
Let 
\begin{equation}
\label{Eq:Lie-GWk}
\Lie(G_{W(k)})=\bigoplus_{n\in\ZZ}\Fg_n
\end{equation}
be the weight decomposition with respect to 
the action of $\Gm$ by $\Ad(\mu^{-1})$.
The display datum $(G,\mu)$ 
is called $1$-bounded if $\Fg_n$ is zero for $n\ge 2$.
\end{Defn}

We note that 
$(G,\mu)$ is $1$-bounded iff $\Lie U^+=\Fg_1$.
A display datum $(G,\mu)$ with reductive $G$
is $1$-bounded iff $\Fg=\Fg_{-1}\oplus\Fg_0\oplus\Fg_1$,
i.e.\ $\mu$ is minuscule.
In this case we call $(G,\mu)$ a minuscule display datum.

\begin{Lemma}
\label{Le:log-minuscule}
If $(G,\mu)$ is a $1$-bounded display datum,
there is a unique $\Gm$-equivariant morphism of schemes
$\log^+:U^+\to V(\Fg_1)$ which induces the identity on the Lie algebras,
and\/ $\log^+$ is an isomorphism of group schemes.
\end{Lemma}

\begin{proof}
We have to show that the morphism $\log^+$ of Lemma \ref{Le:log-U} 
is unique and a group homomorphism.
Let $U^+=\Spec B$ and $I=B_{>0}$.
Then $\log^+$ corresponds to a homomorphism of graded
$W(k)$-modules $s:I/I^2\to I$ 
which is a section of $\pi:I\to I/I^2$.
But $\pi$ is bijective in degree $1$,
and $I/I^2$ is concentrated in degree $1$ since $\mu$ is $1$-bounded,
hence $s$ is unique. 
A similar argument shows that $\log^+$ is a group homomorphism.
%
\end{proof}

\begin{Remark}
\label{Rk:decomp-minuscule}
Let $(G,\mu)$ be a $1$-bounded display datum 
and let $\u S$ be a frame over $W(k)$ for $R$.
Then the decomposition of Proposition \ref{Pr:decomp} takes the form
\[
G(S)_\mu\cong P^-(S)_\mu\times U^+(S)_\mu\cong 
P^-(S_0)\times(\Fg_1\otimes_{W(k)} S_1).
\]
If we write $U^+(tS_1)=\Ker(U^+(S_0)\to U^+(R))$,
it follows that the image of $\tau:G(S)_\mu\to G(S_0)$
is equal to
\[
G(S_0)\times_{G(R)}P^{-}(R)
\cong P^-(S_0)\times U^+(tS_1)\cong
P^-(S_0)\times(\Fg_1\otimes_{W(k)} tS_1).
\]
Let $J_n$ denote the kernel of $t^n:S_n\to S_0$. 
We get an exact sequence
\[
0\to\Fg_1\otimes_{W(k)} J_1\to G(S)_\mu\xrightarrow{\;\tau\;}
G(S_0)\times_{G(R)}P^{-}(R)\to 0.
\]
We have $J_1=0$
for the frames $\u W(R)$, $\u W(B/A)$, $\u\WW(R)$, and $\u\WW(B/A)$,
but not for the truncated Witt frame $\uW m(R)$.
\end{Remark}

\begin{Remark}
\label{Rk:BP}
For a reductive minuscule display datum,
$G$-displays of type $\mu$ over $\u W(R)$ 
coincide with the $(G,\mu^{-1})$-displays over $R$ of \cite{Bueltel-Pappas}. 
Indeed, the display group $G(\u W(R))_\mu$
coincides with the group $H^{\mu^{-1}}(R)$ of loc.cit.\
by Remark \ref{Rk:decomp-minuscule},
and the homomorphism $\sigma:G(\u W(R))_\mu\to G(W(R))$ coincides with
the homomorphism $\Phi_{G,\mu^{-1}}$ of \cite[Proposition 3.1.2]{Bueltel-Pappas}
by Lemma \ref{Le:action-tf}.
We note that the existence of this homomorphism requires some work in loc.cit., while here it is simply induced by the ring homomorphism $\sigma$.
\end{Remark}

\begin{Example}
\label{Ex:orth-minuscule}
Let $G=\Ogr(\psi)=\{g\in\GL_n\mid g^tJg=J\}$ as in \S \ref{Se:orth-disp}.
Let $\mu=(1,0,\ldots,0,-1)\in\ZZ^d$ and let $\mu:\Gm\to G$
be the associated cocharacter,
$\mu(x)=\diag(x^{\mu_1},\ldots,x^{\mu_n})$.
Then $\mu$ is minuscule.
Explicitly, $L$ is the group of block matrices
\[
\left(
\begin{matrix}
a & 0 & 0 \\
0 & B & 0 \\
0 & 0 & a^{-1}
\end{matrix}
\right)
\]
with $a\in\Gm$ and $B\in\Ogr(\psi')$, where $\psi'$ is the analogue
of $\psi$ in dimension $n-2$. 
The group $U^+$ is the group of block matrices
\[
g=
\left(
\begin{matrix}
1 & 0 & 0 \\
x & E_{n-2} & 0 \\
z & y & 1
\end{matrix}
\right)
\]
with $x+y^t=0$ and $2z=yx$; 
here $x$ and $y$ are a column and a row vector of length $n-2$.
The isomorphism $\log^+:U^+\xrightarrow\sim\GG_a^{n-2}$ is given by $g\mapsto x$. 
\end{Example}

\section{Deformations}

The classical deformation theory of $1$-displays over the small Witt vector ring
can be formulated as follows. 
For a PD thickening $B\to A$ of admissible Artin rings 
(Example \ref{Ex:frame-WW}) 
we have
frame homomorphisms 
\[
\u\WW(B)\xrightarrow{\;\alpha\;}\u\WW(B/A)\xrightarrow{\;\varepsilon\;}\u\WW(A)
\]
given by functoriality.
Then $\varepsilon$ induces an equivalence of $1$-displays,
i.e.\ every $\u\WW(A)$-display lifts uniquely to $\u\WW(B/A)$,
and lifts of $1$-displays under $\alpha$ correspond to lifts of
the Hodge filtration.
This generalizes as follows.
Lifts of (higher) displays and of $G$-displays under $\alpha$
again correspond to lifts of a Hodge filtration,
but $\varepsilon$ induces an equivalence only for $1$-bounded $G$-displays.
We will verify these facts in an abstract setting.

\subsection{A unique lifting lemma for $1$-bounded displays}

Let 
$\varepsilon:\u{\tilde S}\to\u S$
be a homomorphism of frames for the same ring $R$, 
i.e.\ the induced ring homomorphism 
$\tilde S_0/t\tilde S_1\to S_0/tS_1$ is bijective,
and let $K\subseteq\tilde S$
be the kernel of $\varepsilon$. 
Then $K=\bigoplus_nK_n$ with $K_n=K\cap \tilde S_n$.
We consider the commutative diagram with exact rows
\[
\xymatrix@M+0.2em{
0 \ar[r] &
K_1 \ar[r] \ar[d]^{\sigma_1} &
\tilde S_1 \ar[r]^{\varepsilon_1} \ar[d]^{\sigma_1} &
S_1 \ar[d]^{\sigma_1} \\
0 \ar[r] &
K_0 \ar[r] &
\tilde S_0 \ar[r]^{\varepsilon_0} &
S_0 
}
\]
and its variant with $\tau_1$ in place of $\sigma_1$;
here $\tau_1$ is the multiplication by $t$.

\begin{Defn}
\label{Def:sigma-nilpotent-ext}
A frame homomorphism $\varepsilon$ as above
is called a $1$-thickening if $\varepsilon_n:\tilde S_n\to S_n$ 
is surjective for $n=0,1$
and $\tau_1:K_1\to K_0$ is bijective.
In this case let
\[
\dot\sigma:=\sigma_1\circ\tau_1^{-1}:K_0\to K_0.
\]
A $1$-thickening of frames $\varepsilon$ is called 
a $\dot\sigma$-nilpotent thickening if $\dot\sigma$
is pointwise nilpotent on $K_0/pK_0$ and $K_0$ is $p$-adically complete.
If $\u S$ and $\u{\tilde S}$ are $p$-adic frames,
$\varepsilon$ is called a stable $1$-thickening 
(stable $\dot\sigma$-nilpotent thickening) 
if for each etale homomorphism $R\to R'$ the associated base change of
$\varepsilon$ as in \S \ref{Se:etale-loc} is a $1$-thickening 
($\dot\sigma$-nilpotent thickening).
\end{Defn}

\begin{Remark}
The relation $\sigma(t)=p$ implies that $p\dot\sigma=\sigma_0$ on $K_0$.
\end{Remark}


\begin{Example}
\label{Ex:WBA-sigma-nil-thickening}
For a PD thickening of $p$-adic rings $B\to A$ with kernel $J$ the
natural frame homomorphism $\u W(B/A)\to\u W(A)$ is a $1$-thickening.
In this case $K_n=W(J)$ for $n\ge 0$, 
the homomorphism $\tau_1:K_{1}\to K_0$ is the identity of $W(J)$, 
and under the isomorphism $\log:W(J)\cong J^\NN$ 
the endomorphism $\dot\sigma$ of $W(J)$ acts on $J^\NN$ by the shift
$(a_0,a_1,\ldots)\mapsto(a_1,a_2,\ldots)$.
If $B\to A$ is a PD thickening of admissible local Artin rings 
(Example \ref{Ex:frame-WW}) with nilpotent divided powers,
then $\u\WW(B/A)\to\u\WW(A)$ is a stable $\dot\sigma$-nilpotent thickening because
the shift is pointwise nilpotent on $J^{(\NN)}$, 
and an etale base change of $\u\WW(B/A)$ is a frame of the same type 
by Lemma \ref {Le:frame-WW-bc}.
\end{Example}

If $(G,\mu)$ is a display datum 
and if $H\subseteq G_{W(k)}$ is a $\Gm$-invariant subgroup scheme, 
for $K=\ker(\varepsilon)$ as above we set
\[
H(K)_\mu=\Ker(H(\tilde S)_\mu\to H(S)_\mu), 
\]
\[
H(K_0)=\Ker(H(\tilde S_0)\to H(S_0)).
\]

\begin{Lemma}
\label{Le:Hmu}
Let $(G,\mu)$ be a $1$-bounded display datum, 
let $\varepsilon:\u{\tilde S}\to\u S$ 
be a $1$-thickening of frames over $W(k)$,
and let $H$ be one of the group schemes $G_{W(k)},P^\pm,U^\pm,L$.
Then:
\begin{enumerate}
\item[(a)]
\label{It:HKmu-HK0}
$\tau:\tilde S\to\tilde S_0$ induces
an isomorphism $H(K)_\mu\cong H(K_0)$.
\item[(b)]
\label{It:HtS-HS}
The homomorphisms $H(\tilde S)_\mu\to H(S)_\mu$ and $H(\tilde S_0)\to H(S_0)$ 
induced by $\varepsilon$ are surjective.
\end{enumerate}
\end{Lemma}

\begin{proof}
Since $K_0\subseteq\Rad(\tilde S_0)$, 
the multiplication $P^-(K_0)\times U^+(K_0)\to G(K_0)$ is bijective,
moreover $P^-(K)_\mu\times U^+(K)_\mu\to G(K)_\mu$ is bijective by
Proposition \ref{Pr:decomp}. 
Hence to prove (a) we can assume that $H\subseteq P^-$ or $H=U^+$.
For $H\subseteq P^-$ the assertion holds since $K_{\le 0}=K_0[t]$.
For $H=U^+$ we have $H(K_0)=\Lie(U^+)\otimes K_0$ 
and $H(K)_\mu=\Lie(U^+)\otimes K_1$ since $\mu$ is $1$-bounded,
and the assertion holds since $t:K_1\to K_0$ is bijective.
To prove (b) we note that $K_0/pK_0$ is a nil-ideal 
because for $x\in K_0$ there is a $y\in\tilde S_0$ 
with $x^p=\sigma_0(x)+py=p(\dot\sigma(x)+y)$ and thus $x^{p+1}\in pK_0$.
Clearly the ideal $p^iK_0/p^{i+1}K_0$ is nilpotent for $i\ge 1$.
Since $H$ is smooth and $K_0$ is $p$-adically complete
it follows that $H(\tilde S_0)\to H(S_0)$ is surjective.
To prove that $H(\tilde S)_\mu\to H(S)_\mu$ is surjective we
can assume again that $H\subseteq P^-$ or $H=U^+$.
In the first case we use Lemma \ref{Le:tau-GSmu-GS0},
in the second case we look at 
$\Lie(U^+)\otimes\tilde S_1\to\Lie(U^+)\otimes S_1$,
which is surjective.
\end{proof}

\begin{Prop}
[Unique lifting lemma]
\label{Pr:deformation-lemma}
Let $(G,\mu)$ be a $1$-bounded display datum. 
If\/ $\varepsilon:\u{\tilde S}\to\u S$ is a stable
$\dot\sigma$-nilpotent thickening of $p$-adic frames over $W(k)$
for $R$, 
the base change functor 
\begin{equation}
\label{Eq:def}
\varepsilon:G\GDisp_\mu(\u{\tilde S})\to G\GDisp_\mu(\u S)
\end{equation}
is an equivalence of categories.
\end{Prop}

\begin{proof}
By definition, 
\eqref{Eq:def} is the functor of etale quotient groupoids
\[
[G(\tilde\bS_0)/G(\tilde\bS)_\mu]\to[G(\bS_0)/G(\bS)_\mu]
\]
with respect to the action \eqref{Eq:action}.
We will show that even the functor 
of presheaf quotient groupoids is an equivalence, 
i.e.\ that the functor 
\begin{equation}
\label{Eq:groupoids}
[G(\tilde S_0)/G(\tilde S)_\mu]\to[G(S_0)/G(S)_\mu]
\end{equation}
is an equivalence. 
Since the maps $G(\tilde S_0)\to G(S_0)$ and $G(\tilde S)_\mu\to G(S)_\mu$
are surjective by Lemma \ref{Le:Hmu} (b), we have to show that 
the action of $G(K)_\mu$ on the fibres of $G(\tilde S_0)\to G(S_0)$ 
is simply transitive,
i.e.\ that for $g\in G(\tilde S_0)$ and $h\in G(K_0)$
there is a unique $z\in G(K)_\mu$ such that
\begin{equation}
\label{Eq:Cond}
\tau(z)^{-1}g\sigma(z)=hg.
\end{equation}
Since $\tau:G(K)_\mu\to G(K_0)$ is bijective by Lemma \ref{Le:Hmu} (a),
we can define a group homomorphism
\begin{equation}
\label{Eq:UUUg}
\UUU_g:G(K_0)\to G(K_0),\qquad y\mapsto g(\sigma(\tau^{-1}(y)))g^{-1},
\end{equation}
and \eqref{Eq:Cond} is equivalent to $y^{-1}\UUU_g(y)=h$ for $y=\tau(z)$.
Thus 
we have to show that for each $g\in G(\tilde S_0)$ the map
\begin{equation}
\label{Eq:bijective}
G(K_0)\to G(K_0),\qquad y\mapsto \UUU_g(y)^{-1}y
\end{equation}
is bijective.
Since $K_0$ is $p$-adically complete and $G$ is smooth,
$G(K_0)$ is the limit over $m$ of $G(K_0/p^mK_0)$,
and the reduction map $G(K)\to G(K_0/p^mK_0)$ is surjective;
let $G(p^mK_0)$ denote its kernel.

\begin{Lemma}
\label{Le:UUUg}
The endomorphism $\UUU_g$ 
of \eqref{Eq:UUUg} preserves $G(p^mK_0)$
and acts on $G(K_0/p^mK_0)$ as a pointwise nilpotent operator.
\end{Lemma}

It follows that \eqref{Eq:bijective} is surjective
because for $h\in G(K_0)$
the converging product $y=\cdots\UUU_g^2(h)\UUU_g(h)h$ 
satisfies $\UUU_g(y)^{-1}y=h$, 
and \eqref{Eq:bijective} is injective since 
$\UUU_g(y)^{-1}y=\UUU_g(y')^{-1}y'$ implies that
$y'y^{-1}$ is fixed by $\UUU_g$ and thus equal to $1$.
This proves Proposition \ref{Pr:deformation-lemma} 
modulo Lemma \ref{Le:UUUg}.
\end{proof}

\begin{proof}[Proof of Lemma \ref{Le:UUUg}]
The endomorphism $\sigma\circ\tau^{-1}$ 
preserves the decomposition $G(K_0)=P^-(K_0)\times U^+(K_0)$ 
and acts on $U^+(K_0)=(\Lie U^+)\otimes K_0$ by ${\id}\otimes\dot\sigma$.
Let $P^-=\Spec A^-$ as in \eqref{Eq:Ppm} 
and let $\nu:A^-\to A^-$ be the ring homomorphism 
given by $\nu(a)=p^ia$ when $\deg(a)=-i$.
Then $\sigma\circ\tau^{-1}$ acts on $P^-(\tilde S_0)$
by sending $x:A^-\to\tilde S_0$ to $\sigma_0\circ x\circ\nu$.
It follows that $\sigma\circ\tau^{-1}$ and thus $\UUU_g$ 
preserve $G(p^mK_0)$, and it suffices to verify that $\UUU_g$ 
is pointwise nilpotent on $G(\bar K_m)$ 
with $\bar K_m=p^mK_0/p^{m+1}K_0$ for $m\ge 0$. 
Now $\sigma\circ\tau^{-1}$ is zero on $P^-(\bar K_m)$ 
because $\sigma_0(a)=p\dot\sigma(a)$ for $a\in K_0$.
Hence the endomorphism $\UUU_g$ of $G(\bar K_m)$ factors as
\begin{equation}
G(\bar K_m)
\xrightarrow{\;pr\;}U^+(\bar K_m)
\xrightarrow{\;{\id}\otimes\dot\sigma\;}U^+(\bar K_m)\xrightarrow{y\mapsto gyg^{-1}}
G(\bar K_m),
\end{equation}
which is pointwise nilpotent iff the cyclic permutation of these maps
\begin{equation}
\label{Eq:End-of-U+Lm}
U^+(\bar K_m)\xrightarrow{y\mapsto gyg^{-1}}G(\bar K_m)
\xrightarrow{\;pr\;}U^+(\bar K_m)
\xrightarrow{\;\id\otimes\dot\sigma\;}U^+(\bar K_m)
\end{equation}
is pointwise nilpotent. 
Here $y\mapsto pr(gyg^{-1})$ is an endomorphism of the pointed set
$U^+(\bar K_m)$ which is given by some power series over $\tilde S_0$. 
The terms of higher order are annihilated by ${\id}\otimes\dot\sigma$
because $\dot\sigma(ab)=p\dot\sigma(a)\dot\sigma(p)$ for $a,b\in K_0$.
Hence the endomorphism \eqref{Eq:End-of-U+Lm} of
$U^+(\bar K_m)=(\Lie U^+)_{\tilde S_0}\otimes_{\tilde S_0}\bar K_m$
takes the form 
$\psi\otimes\dot\sigma$ for a $\sigma_0$-linear endomorphism $\psi$.
Since $\dot\sigma$ is pointwise nilpotent on $\bar K_m$
it follows that \eqref{Eq:End-of-U+Lm} is pointwise nilpotent.
\end{proof}

\begin{Remark}
One verifies easily that
the $\sigma_0$-linear endomorphism $\psi$ of $(\Lie U^+)_{\tilde S_0}$
in the end of the proof 
is given by $(1\otimes\sigma_0)\circ pr\circ\Ad(g)$.
\end{Remark}

\begin{Remark}
The essence of the proof of Proposition \ref{Pr:deformation-lemma} is 
very similar to the proof of \cite[Theorem 3.5.4]{Bueltel-Pappas}, 
which is the key to the deformation theory of minuscule $G$-displays 
over Witt vectors in \cite{Bueltel-Pappas}. 
In that context, a nilpotence condition is needed because for a
PD thickening of $B\to A$ of $W(k)$-algebras in which $p$ is nilpotent 
the frame homomorphism $\u W(B/A)\to\u W(A)$ is a $1$-thickening
which is not in general $\dot\sigma$-nilpotent,
and therefore the functor $G\GDisp_\mu(\u W(B/A))\to G\GDisp_\mu(\u W(A))$
is not an equivalence. 
But the restriction of this functor to $G$-displays
which are adjoint nilpotent in the sense of \cite{Bueltel-Pappas} 
is an equivalence.
Indeed, a $G$-display is adjoint nilpotent in the sense of loc.cit.\ iff the
above endomorphism $\psi$ is nilpotent; see \cite[Lemma 3.4.4]{Bueltel-Pappas};
in that case the endomorphism $\UUU_g$ is nilpotent.
\end{Remark}

\subsection{The Hodge filtration of a display}
\label{Se:Hodge-fil-disp}

Let $\u S$ be a frame for $R$ 
and let $\bar\tau:S\to R$ be the composition of
$\tau:S\to S_0$ and the projection $S_0\to R$.
For a finite projective graded $S$-module $M$ we define
\[
M_R=M^\tau\otimes_{S_0}R=M\otimes_{S,\bar\tau}R.
\]
The modules 
\[
E_n=\Image(M_n\to M^\tau\to M_R)
\]
form a descending filtration $(E_n)_{n\in\ZZ}$
by locally direct summands of $M_R$,
with $E_n=M_R$ for small $n$ and $E_n=0$ for large $n$.
If $M=L\otimes_{S_0}S$ for a finite projective graded $S_0$-module $L$, 
then $E_n=\bigoplus_{i\ge n}L_i\otimes_{S_0}R$.
We call $(E_n)$ the Hodge filtration of $M_R$.
If $(M,\beta)$ is an orthogonal module over $\u S$ 
(Definition \ref{Def:orth-disp})
then $\beta$ induces a perfect symmetric bilinear form on $M_R$,
and the Hodge filtration is self-dual in the sense that $(E_n)^\perp=E_{1-n}$.

\subsection{Lifts of the Hodge filtration}
\label{Se:lifts-Hodge}

Let $\alpha:\u S'\to\u S$ be a homomorphism of frames such that 
$\alpha_0:S'_0\to S_0$ is bijective. 
We identify $S_0=S_0'$. Let $R'=S_0/tS'_1$ and $R=S_0/tS_1$,
thus $R=R'/I$ with $I=tS_1/tS_1'$.
The composition $S\xrightarrow\tau S_0= S_0'\to R'$ 
is a ring homomorphism $\tilde\tau:S\to R'$,
and for a finite projective graded $S$-module $M$ 
we can define an $R'$-module 
\[
M_{R'}=M^\tau\otimes_{S_0}R'=M\otimes_{S,\tilde\tau}R'.
\]
Then $M_R=M_{R'}\otimes_{R'}R$.
If $M'$ is a finite projective graded $S'$-module and $M=M'\otimes_{S'}S$,
then $M_{R'}=M'_{R'}$, and 
the Hodge filtration $(E'_n)$ of $M'_{R'}$ is a lift of the Hodge filtration
$(E_n)$ of $M_R$.

\begin{Prop}
\label{Pr:Hodge-lift}
Assume that for $n\ge 1$ the homomorphism 
$\alpha_n:S'_n\to S_n$ is injective
and $t^n:S_n\to S_0$ induces an isomorphism $S_n/S'_n\xrightarrow\sim I$,
and that finite projective $R$-modules can be lifted to $R'$.
Then the functor from finite projective graded $S'$-modules 
(or displays over $\u S'$)
to finite projective graded $S$-modules $M$ (or displays $\u M$ over $\u S$) 
together with a lift of the Hodge filtration of $M_R$ to a filtration of $M_{R'}$ 
by direct summands is an equivalence.
\end{Prop}

\begin{proof}
It suffices to treat the case of modules  
because for a finite projective graded $S'$-module $M'$ and $M=M'\otimes_{S'}S$ 
we have $M^\tau=M'^\tau$ and $M^\sigma=M'^\sigma$, 
so display structures on $M$ and on $M'$ are the same.

For a finite projective graded $S'$-module $M'$ 
with Hodge filtration $(E_n')$ and $M=M'\otimes_{S'}S$, 
the submodule $M'_n\subseteq M_n$ is the inverse image
of $E_n'\subseteq M'_{R'}=M_{R'}$
under $M_n\to M^\tau\to M_{R'}$. 
Indeed, this assertion holds for $M'=S'(d)$ 
by the assumptions on $\u S'\to\u S$,
and then it holds for every $M'$ 
since the property passes to direct summands.
It follows that the functor of the proposition is fully faithful.

Let us verify that the functor is essentially surjective.
If $M$ over $S$ with a lift of the Hodge filtration $(E'_n)$ is given, 
we necessarily define $M'_n\subseteq M_n$ 
as the inverse image of $E'_n\subseteq M_{R'}$.
Then $M'=\bigoplus M_n'$ is a graded $S'$-module.
We have to verify that
$M'$ is finite projective and that $M'\otimes_{S'}S\to M$ is bijective.
We can replace $M$ by $M\oplus N$ for another finite projective graded $S$-module
$N$ and thus assume that $M=L\otimes_{S_0}S$ for a finite projective graded
$S_0$-module $L$; note that the Hodge filtration of $N$ can be lifted to $R'$
by the assumption. The assumptions on $\u S'\to\u S$ imply that the kernel
of $M_n\to M_R$ maps surjectively onto $IM_{R'}$, 
hence $M_n'\to E'_n$ is surjective.
It follows that the projection $L\to L\otimes_{S_0}R=M\otimes_{S,\rho}R$
can be lifted to a homomorphism of graded $S_0$-modules $L\to M$ with
image in $M'$. The resulting homomorphism $L\otimes_{S_0}S\to M$ is an
isomorphism by Corollary \ref{Co:NAK-proj}, and it induces an isomorphism
$L\otimes_{S_0}S'\cong M'$ by the first part of the proof.
\end{proof}

\begin{Cor}
\label{Co:Hodge-lift-orth}
In the situation of Proposition \ref{Pr:Hodge-lift}, 
orthogonal displays (or orthogonal graded modules) over $\u S'$ 
are equivalent to orthogonal displays (or orthogonal graded modules) over
$\u S$ together with a self-dual lift of the Hodge filtration to ${R'}$.
\end{Cor}

\begin{proof}
Let $(M,\beta)$ be an orthogonal graded module over $\u S$ and let $M'$ be
a lift of $M$ to $S'$. 
The isomorphism $\beta':M\to M^*$ corresponding to $\beta$
lifts to an isomorphism $M'\to M'^*$ iff $\beta'$ preserves the Hodge filtration,
which means that the Hodge filtration of $M'$ is self-dual with respect to $\beta$.
\end{proof}

\begin{Example}
For a PD thickening of $p$-adic rings $B\to A$,
the frame homomorphism $\u W(B)\to \u W(B/A)$ satisfies the hypotheses
of Proposition \ref{Pr:Hodge-lift}.
Similarly, if $B\to A$ is a PD thickening of admissible local Artin rings
(Example \ref{Ex:frame-WW}) with nilpotent divided powers,
$\u\WW(B)\to\u\WW(B/A)$ satisfies the hypotheses
of Proposition \ref{Pr:Hodge-lift}.
This is easily verified.
\end{Example}

\subsection{The Hodge filtration of a $G$-display}

Let $(G,\mu)$ be a display datum.
The Hodge filtration of displays has a natural counterpart for $G$-displays.
We will only treat the case of $G$-displays over $p$-adic frames as
in \S \ref{Se:G-disp-nil-frame} and leave out the obvious modifications 
for $G$-displays over Witt vectors as in \S \ref{Se:G-disp-W}.

If $\u S$ is a frame over $W(k)$ for $R$,
the ring homomorphism $\bar\tau:S\to R$ 
of \S \ref{Se:Hodge-fil-disp} induces a group homomorphism 
\[
\bar\tau:G(S)_\mu\to G(R)
\]
with image in $P^-(R)$; see Proposition \ref{Pr:decomp}.
Let $\bar\tau_0:G(S)_\mu\to P^-(R)$ be the resulting group homomorphism.
Assume that $\u S$ is a $p$-adic frame.
Then $\bar\tau$ and $\bar\tau_0$ are homomorphisms of etale sheaves over $\Spec R$.
For an etale $G(\bS)_\mu$-torsor $Q$ we can form the $G_R$-torsor
$Q_R=Q^{\bar\tau}$ and the $P^-_R$-torsor
\[
Q_0=Q^{\bar\tau_{0}}\subseteq Q_R,
\]
which will be called the Hodge filtration of $Q_R$.
Note that a $G$-display over $\u S$ has an underlying $G(\bS)_\mu$-torsor $Q$; cf.\ Remark \ref{Rk:G-disp}.

Now let $\alpha:\u S'\to\u S$ be a homomorphism 
of $p$-adic frames
over $W(k)$
with $S'_0=S_0$ as in \S \ref{Se:lifts-Hodge}.
Then $I=\Ker(R'\to R)$ is a nil-ideal
since this holds for the kernel of $S_0/p\to R/p$ 
by Remark \ref{Rk:S0pRp-nil}.
Thus the categories of affine etale
schemes over $R$ and over $R'$ are equivalent,
and  we can consider $\u S$ and $\u S'$ as sheaves on the same site.

The ring homomorphism $\tilde\tau:S\to R'$ 
induces a group homomorphism $\tilde\tau:G(S)_\mu\to G(R')$,
and for a $G(\bS)_\mu$-torsor $Q$ 
we can form the etale $G_{R'}$-torsor $Q_{R'}=Q^{\tilde\tau}$.
If $Q'$ is an etale $G(\bS')_\mu$-torsor and $Q$ the induced
$G(\bS)_\mu$-torsor then $Q_{R'}=Q'_{R'}$,
and the Hodge filtration $Q_0'\subseteq Q'_{R'}$
is a lift of the Hodge filtration $Q_0\subseteq Q_R$
in the sense that the reduction map $Q'_{R'}\to Q_R$
induces $Q_0'\to Q_0$.

\begin{Prop}
\label{Pr:Hodge-lift-G}
Assume that for each $n\ge 1$ with $\Fg_n\ne 0$ in 
\eqref{Eq:Lie-GWk} the homomorphism $\alpha_n:S'_n\to S_n$ is injective,
and $t^n:S_n\to S_0$ 
induces an isomorphism $S_n/S_n'\cong I$.
Then the functor from etale $G(\bS')_\mu$-torsors
(or $G$-displays of type $\mu$ over $\u S'$) to etale $G(\bS)$-torsors $Q$
(or $G$-displays $\u Q$ of type $\mu$ over $\u S$)
together with a lift of the Hodge filtration of $Q_R$ to 
an etale $P^-_{R'}$-torsor in $Q_{R'}$ is an equivalence.
\end{Prop}

The proof of Proposition \ref{Pr:Hodge-lift-G} is based on the following lemma.

\begin{Lemma}
\label{Le:Hodge-lift-G}
Under the assumptions of Proposition \ref{Pr:Hodge-lift-G}
there is a cartesian square with injective horizontal maps and
surjective vertical maps:
\[
\xymatrix@M+0.2em{
G(S')_\mu \ar[r] \ar[d] &
G(S)_\mu \ar[d] \\
P^-(R') \ar[r] & G(R')\times_{G(R)}P^-(R)
}
\]
\end{Lemma}

\begin{proof}
The vertical maps are induced by 
$\tilde\tau:G(S)_\mu\to G(R')$.
The multiplication $P^-(S)_\mu\times U^+(S)_\mu\to G(S)_\mu$
and its analogue for $S'$ are bijective.
Similarly the lower right corner of the diagram is bijective to
$P^-(R')\times U^+(I)$ where $U^+(I)$ denotes the kernel
of $U^+(R')\to U^+(R)$. 
Using that $P^-(S')_\mu\cong P^-(S_0)\cong P^-(S)_\mu$
by Lemma \ref{Le:tau-GSmu-GS0},
the lemma is an easy consequence of Lemma \ref{Le:log-U}.
\end{proof}

\begin{proof}[Proof of Proposition \ref{Pr:Hodge-lift-G}]
It suffices to treat the case of torsors because for an etale
$G(\bS')_\mu$-torsor $Q'$ and the induced etale $G(\bS)_\mu$-torsor $Q$
we have $Q^\sigma=Q'^\sigma$ and $Q^\tau=Q'^\tau$ as etale
torsors for $G(\bS_0)=G(\bS'_0)$, so display structures on
$Q$ and on $Q'$ are the same; see Remark \ref{Rk:G-disp}.

We claim that the hypothesis of Proposition \ref{Pr:Hodge-lift-G}
is preserved under etale base change.
The hypothesis is an exact sequence $0\to S'_n\to S_n\to I\to 0$
for certain $n$. 
For large $m$ so that $p^mR'=0$ this gives
$I\to S'_n/p^m\to S_n/p^m\to I\to 0$.
These sequences for a projective system with respect to $m$,
where the transition maps on the first module $I$ are eventually zero.
This property are preserved under etale base change,
and the claim follows.

Lemma \ref{Le:Hodge-lift-G} implies 
that the functor of the proposition is fully faithful.
To prove that the functor is essentially surjective let
a $G(\bS)_\mu$-torsor $Q$ and a lift $Q_0'\subseteq Q_{R'}$
of the Hodge filtration $Q_0\subseteq Q_{R}$ be given.
Let $Q'\subseteq Q$ be the inverse image of $Q_0'\subseteq Q_{R'}$
as an etale sheaf with an action of $G(\bS')_\mu$.
Etale locally, $Q$ and $Q_0'$ are trivial.
Since $Q_0'$ is a lift of $Q_0$ and since
$G(\bS)_\mu\to G(R')\times_{G(R)}P^-(R)$ is surjective,
there are compatible local trivializations of $Q$ and $Q_0'$,
and it follows that $Q'$ is a $G(\bS')_\mu$-torsor.
\end{proof}

\section{K3 displays and varieties}
\label{Se:K3-disp}

\begin{Defn}
A K3 display over a frame $\u S$ is a display $\u M$ of type $(2,1,\ldots,1,0)$
with a perfect symmetric bilinear form $\beta:\u M\otimes\u M\to\u S(-2)$.
\end{Defn}

\begin{Remark}
\label{Rk:K3-disp}
The bilinear form $\beta$ is equivalent to $\u M(1)\otimes\u M(1)\to\u S$,
so K3 displays are equivalent to orthogonal displays 
of type $\mu=(1,0,\ldots,0,-1)$,
which are equivalent to $G$-displays of type $\mu$ for the group
$G=\Ogr(\psi)$ of \S \ref{Se:orth-disp} 
by Proposition \ref{Pr:orth-disp}.
Here $\mu$ is minuscule by Example \ref{Ex:orth-minuscule}.
\end{Remark}

\begin{Remark}
A K3 display $\u M$ over $\u S$ is effective, so the associated $S_0$-module
$M^\tau$ carries a perfect symmetric bilinear form $\beta$
and a $\sigma_0$-linear endomorphism $F_0$ with 
$\beta(F_0(x),F_0(y))=p^2\sigma_0(\beta(x,y))$.
\end{Remark}

In the following let $A$ be a local Artin ring 
with perfect residue field $k$ of characteristic $p\ge 3$,
in particular $A$ is admissible  as in Example \ref{Ex:frame-WW}.
By Langer-Zink \cite{Langer-Zink:K3} the second crystalline cohomology of a
K3 surface $X$ over $A$, 
and more generally of a scheme of K3 type $X$ over $A$,
carries a natural structure of a K3 display over $\u\WW(A)$,
and also over $\u\WW(B/A)$ when $B\to A$ is a nilpotent PD thickening.
Let us recall this in some detail.
The only new aspect is the use of the unique lifting lemma
Proposition \ref{Pr:deformation-lemma}.

\subsection{K3 surfaces}
\label{Se:K3-surf}

We begin with the case of K3 surfaces.
A K3 surface over $A$ is a smooth proper scheme $X$ over $A$ such
that the special fibre $X_k$ is a K3 surface.
Since $\WW(B/A)\to A$ is a $p$-adic PD thickening, 
the crystalline cohomology group
\[
H=H_{B/A}(X)=H^2_{\crys}(X/\WW(B/A))
\]
is defined. It is a free module over $\WW(B/A)$ of rank $22$ 
with a $\sigma$-linear map $F:H\to H$,
and the cup product is a perfect symmetric bilinear form $\gamma$ 
on $H$ with $\gamma(Fx,Fy)=p^2\sigma(\gamma(x,y))$.

The module $H$ carries a display structure which is natural in $X$ and in $B\to A$, 
i.e.\ there is a K3 display $\u M=\u H{}_{B/A}(X)$
over $\WW(B/A)$ with an isomorphism $M^\tau\cong H$
which preserves the Frobenius and the quadratic form.
In the case $B=A$ we write $\u H{}_{A/A}(X)=\u H{}_A(X)$.
The K3 display $\u H{}_{B/A}(X)$ is constructed as follows.
Let $R=W(k)[[t_1,\ldots,t_{20}]]$ be the universal deformation ring
of $X_k$ and let $Y$ over $\Spf R$ be the universal deformation.
Let $\sigma_R:R\to R$ be the Frobenius lift defined by $t_i\mapsto t_i^p$
and let $\u S$ be the tautological frame of 
Example \ref{Ex:frame-tautological} defined by $S=R[t]$.
The cohomology group
\[
H^2_{DR}(Y/R)=H^2_{\crys}(Y/R)
\]
with its Hodge filtration and Frobenius 
defines a K3 display $\u M(Y)$ over $\u S$;
see Example \ref{Ex:disp-tautological}.
The Lazard homomorphism $R\to W(R)$ corresponding to $\sigma_R$ 
(see \cite[Chap.~VII, Prop.~4.12]{Lazard})
induces a frame homomorphism $\u S\to\u\WW(R)$.
The choice of a lift of $X$ to $B$ gives a ring homomorphism $R\to B$.
Now $\u H{}_{B/A}(X)$ is defined as the base change of $\u M(Y)$ under the
frame homomorphism $\u S\to\u\WW(R)\to\u\WW(B)\to\u\WW(B/A)$.
One has to show that this is independent of the chosen lift
of $X$ to $B$, which follows from \cite[Corollary 20]{Langer-Zink:K3}
and its proof.

The homomorphism $\u\WW(B/A)\to\u\WW(A)$ induces an equivalence of K3 displays
by Proposition \ref{Pr:deformation-lemma} and Remark \ref{Rk:K3-disp},
so $\u H{}_{B/A}(X)$ is the unique lift of $\u H{}_A(X)$.
The construction of $\u H{}_{B/A}(X)$ explained before implies 
that the underlying module is given by the crystalline cohomology of $X$,
which is essential for the following result.

\begin{Thm}
\label{Th:K3-surface}
Let $B\to A$ be a surjective homomorphism of local Artin rings 
with perfect residue field $k$ of characteristic $p\ge 3$ and
let $X$ be a K3 surface over $A$.
Then deformations of $X$ to a K3 surface over $B$ are
equivalent to deformations of the K3 display $\u H{}_A(X)$
over $\u\WW(A)$ to a K3 display over $\u\WW(B)$.
\end{Thm}

If $X$ is ordinary this is proved in \cite{Langer-Zink:K3}.

\begin{proof}
By induction
we can assume that the kernel of $B\to A$ has square zero,
so $B\to A$ is a PD thickening with nilpotent divided powers. 
Let $\u M=\u H{}_{B/A}(X)$. Then
\[
M_A=M^\tau\otimes_{\WW(B)} A\cong H^2_{\crys}(X/A)=H^2_{DR}(X/A)
\]
and the Hodge filtration of $M_A$ 
corresponds to the Hodge filtration of $H^2_{DR}$.
Similarly $M_B=M^\tau\otimes_{\WW(B)} B\cong H^2_{\crys}(X/B)$.
By \cite[Th\'eor\`eme 2.1.11]{Deligne:Cristaux-ordinaires} and its proof, 
the set of liftings of $X$ to $B$ is bijective to the set of 
self-dual liftings of the Hodge filtration of $H^2_{DR}(X/A)$
to $H^2_{\crys}(X/B)$.
On the display side, $\u M$ is the unique lift of $\u H{}_A(X)$
to a K3 display over $\u\WW(B/A)$, and by Corollary \ref{Co:Hodge-lift-orth}
liftings of $\u M$ to K3 displays over $\u\WW(B)$ correspond to 
self-dual liftings of the Hodge filtration of $M_A$ to $M_B$. 
The result follows.
\end{proof}

\subsection{Schemes of K3 type}

Let again $A$ be a local Artin ring with perfect residue field $k$
of characteristic $p\ge 3$.
Following Langer-Zink \cite{Langer-Zink:K3}, 
a scheme of K3 type over $A$ is a smooth proper
scheme $X$ over $A$ of dimension $2n$ with the following properties.

\begin{enumerate}
\item
$H^q(X_k,\Omega^p)=0$ for $p+q=1$ or $p+q=3$.
\item
$\dim_kH^q(X_k,\OOO_{X_k}) = 1$  for $q = 0,2$.
\item
$\dim_kH^0(X_k,\Omega^2) = 1$.
\item
Each generator $\sigma\in H^0(X_k,\Omega^2)$ is nowhere degenerate,
each generator $\rho\in H^2(X_k,\OOO_{X_k})$ satisfies $\rho^n\ne 0$
in $H^{2n}$,
and the bilinear form on $H^1(X_k,\Omega^1)$ defined by
$(\omega_1,\omega_2)\mapsto\tr(\omega_1\omega_2\sigma^{n-1}\rho^{n-1})$
is perfect.
\end{enumerate}
The assumptions imply that $X_k$ has a formally smooth universal deformation
space and that $H^q(X,\Omega^p)$ is a free
$A$-module compatible with base change for $p+q=2$.
One can lift $\sigma$ and $\rho$ to $A$ and consider 
$\sigma$ as an element of $H^2_{DR}(X/A)$.
Let $\tau\in H^2_{DR}(X/A)$ be a lift of $\rho$.
Then $\epsilon=\tr(\sigma^n\tau^n)\in A$ is defined.
We also assume that 
\begin{enumerate}
\addtocounter{enumi}{4}
\item
the number $n(n+1)$ is invertible in $k$, and $\epsilon=1$.
\end{enumerate}
The latter can always be achieved when $k$ is algebraically closed. 
In this situation \cite{Langer-Zink:K3} construct a symmetric bilinear form
$\BB_{\sigma,\tau}$ on $H^2_{DR}(X/A)$, called the Beauville-Bogomolov form,
which is well-defined up to an $n$-th root of unity, 
and which allows to extend the reasoning of \S \ref{Se:K3-surf} 
to schemes of K3 type with $\BB_{\sigma,\tau}$ in place of the cup product. 

More precisely, let us fix a scheme of K3 type $X_0$ over $k$ and generators
$\sigma$, $\rho$ as above over $k$ such that $\epsilon=1$.
Assume that $X_0$ can be lifted to a smooth projective scheme over $W(k)$.
For each scheme of K3 type $X$ over $A$ with special fibre $X_k=X_0$
there is a well-defined perfect symmetric bilinear form $\BB_{\sigma,\tau}$ 
on $H^2_{DR}(X/A)$; see \cite[Equation (67)]{Langer-Zink:K3}.
For a PD thickening $B\to A$ with nilpotent divided powers
the bilinear form extends to the crystalline cohomology $H^2_{\crys}(X/B)$,
and liftings of $X$ to $B$ correspond to self-dual liftings of the
Hodge filtration by \cite[Corollary 32]{Langer-Zink:K3}.
Moreover there is a K3 display $\u M=\u H{}_{B/A}(X)$ 
over $\u\WW(B/A)$ with an isomorphism
$M^\tau\cong H^2_{\crys}(X/\WW(B/A))$ that preserves the Frobenius
and the quadratic form, and $\u H{}_{B/A}(X)$ is functorial
in $X$ and in $B\to A$; see \cite[Proposition 35 \& Remark]{Langer-Zink:K3}.

It follows that the proof of Theorem \ref{Th:K3-surface} 
carries over to schemes of K3 type and gives the following result, 
which is proved in \cite{Langer-Zink:K3} when $X$ is ordinary.

\begin{Thm}
\label{Th:Def-K3-scheme}
Let $B\to A$ be a surjective homomorphism of local Artin rings 
with perfect residue field $k$ of characteristic $p\ge 3$ and
let $X$ be a scheme of K3 type over $A$ that satisfies the 
assumptions 1-5 such that $X_k$ can be lifted to a smooth projective
scheme over $W(k)$.
Then deformations of $X$ to a smooth scheme over $B$ 
are equivalent to deformations of the K3 display $\u H{}_A(X)$
over $\u\WW(A)$ to a K3 display over $\u\WW(B)$.
\qed
\end{Thm}

\end{document}